\documentclass[11pt,a4paper, reqno]{amsart}

\usepackage{a4wide}

\usepackage{amsfonts,amsthm,amssymb,amsmath,amscd,amsopn,mathabx}
\usepackage[pdftex]{graphicx}
\usepackage{graphics}
\usepackage[matrix,arrow]{xy}
\usepackage[active]{srcltx}

\usepackage{enumitem, hyperref}\hypersetup{colorlinks}

\makeatletter
\def\reflb#1#2{\begingroup
    #2%
    \def\@currentlabel{#2}%
    \phantomsection\label{#1}\endgroup
}
\makeatother
\usepackage{color}%\textcolor{red}{Test}

\definecolor{darkred}{rgb}{1,0,0} %can change the intensity in [0,1]
\definecolor{darkgreen}{rgb}{0,0.8,0}
\definecolor{darkblue}{rgb}{0,0,1}

\hypersetup{colorlinks,
linkcolor=darkblue,
filecolor=darkgreen,
urlcolor=darkred,
citecolor=darkblue}

\newtheorem{thm}{Theorem}
\numberwithin{thm}{section}
\numberwithin{equation}{section}
\newtheorem{theorem}[thm]{Theorem}
\newtheorem*{theorem*}{Theorem}

\newtheorem{corollary}[thm]{Corollary}
\newtheorem*{corollary*}{Corollary}
\newtheorem{lemma}[thm]{Lemma}

\newtheorem{proposition}[thm]{Proposition}

\newtheorem*{conjecture*}{Conjecture}

\newtheorem*{question*}{Question}

\newtheorem*{definition*}{Definition}

\newtheorem*{definitions*}{Definitions}

\newtheorem*{rem*}{Remark}
\theoremstyle{remark}
\newtheorem{remark}[thm]{Remark}

\newtheorem*{remark*}{Remark}

\newtheorem*{remarks*}{Remarks}
\newtheorem*{example*}{Example}

\newtheorem*{examples*}{Examples}

\newcommand{\R}{\mathbb{R}}
\newcommand{\Z}{\mathbb{Z}}
\newcommand{\Q}{\mathbb{Q}}
\newcommand{\C}{\mathbb{C}}
\newcommand{\N}{\mathbb{N}}
\newcommand{\T}{\mathbb{T}}

\def\CP{{\mathbb C}P}
\def\RP{{\mathbb R}P}
\def\HP{{\mathbb H}P}
\def\CaP{{\text Ca}P}

\newcommand{\ep}{\epsilon}
\newcommand{\ga}{\gamma}

\newcommand{\om}{\omega}
\newcommand{\Om}{\Omega}

\def\maslov{{\mu_\text{Maslov}}}
\def\cz{{\mu}}

\def\mi{{\text{min}}}

\def\lo{{\ell_0}}
\def\l{{\ell}}
\def\mi{{\hat\mu}}
\newcommand{\hmu}{\hat{\mu}}
\def\vk{\vec{k}}
\def\HC{{\mathrm{HC}}}
\def\SH{{\mathrm{SH}}}
\def\CC{{\mathrm{CC}}}
\def\HF{{\mathrm{HF}}}
\newcommand{\const}{{\mathit const}}
\newcommand{\Sp}{\mathrm{Sp}}
\newcommand{\Gr}{\mathrm{Gr}}
\newcommand    \TSp     {\widetilde{\mathrm{Sp}}}
\newcommand   \TSpn    {\widetilde{\mathrm{Sp}}{}^*}
\newcommand{\sign}{\operatorname{sign}}
\newcommand{\sgn}{\operatorname{sgn}}
\newcommand{\p}{\partial}
\newcommand\ff{{\mathfrak f}}
\newcommand\AC{{\mathfrak A}}
\newcommand{\eps}{\epsilon}
\begin{document}

\title[Multiplicity of closed Reeb orbits on
prequantizations]{Multiplicity of closed Reeb orbits on
  prequantization bundles}

\author[Viktor Ginzburg]{Viktor L. Ginzburg}
\author[Ba\c sak G\"urel]{Ba\c sak Z. G\"urel}
\author[Leonardo Macarini]{Leonardo Macarini}

\address{Department of Mathematics, UC Santa Cruz, Santa Cruz, CA
  95064, USA} \email{ginzburg@ucsc.edu} 

\address{Department of Mathematics, University of Central Florida,
  Orlando, FL 32816, USA} \email{basak.gurel@ucf.edu}

\address{Universidade Federal do Rio de Janeiro, Instituto de
  Matem\'atica, Cidade Universit\'aria, CEP 21941-909, Rio de Janeiro,
  Brazil} \email{leomacarini@gmail.com}


\date{\today} 

\thanks{The work is partially supported by NSF grant DMS-1308501 (VG),
  NSF CAREER grant DMS-1454342 (BG) and CNPq, Brazil (LM)}

\bigskip

\begin{abstract} We establish multiplicity results for geometrically
  distinct contractible closed Reeb orbits of non-degenerate contact
  forms on a broad class of prequantization bundles.  The results hold
  under certain index requirements on the contact form and are sharp
  for unit cotangent bundles of CROSS's. In particular, we generalize
  and put in the symplectic-topological context a theorem of Duan,
  Liu, Long, and Wang for the standard contact sphere. We also prove
  similar results for non-hyperbolic contractible closed orbits and
  briefly touch upon the multiplicity problem for degenerate forms. On
  the combinatorial side of the question, we revisit and reprove the
  enhanced common jump theorem of Duan, Long and Wang, and interpret
  it as an index recurrence result.

\end{abstract}

\maketitle

\tableofcontents

\section{Introduction}
\label{sec:intro}
The main theme of this paper is the multiplicity problem for
geometrically distinct contractible closed Reeb orbits of
non-degenerate contact forms satisfying certain index conditions on a
broad class of prequantization bundles. The multiplicity results
established here apply to the unit cotangent bundles of CROSS's
(compact rank one symmetric spaces) for which they are sharp and to
some other prequantization bundles. In particular, we generalize and
put in the symplectic-topological context the main theorem from
\cite{DLLW} on the multiplicity of simple closed Reeb orbits on the
standard contact $S^{2n+1}$. On the combinatorial side of the
question, we revisit and reprove the enhanced common jump theorem from
\cite{DLW2} and interpret it as an index recurrence result along the
lines of the index analysis from~\cite{GG:convex}.

The multiplicity problem for geometrically distinct closed Reeb orbits
originated in Hamiltonian dynamics, going back at least a hundred
years. In its modern form, the question is about establishing a lower
bound, ideally sharp, for the number of such orbits of a contact form
$\alpha$ on a given contact manifold $(M^{2n+1},\xi)$. The form
$\alpha$ is usually required to meet some additional conditions
playing both conceptual and technical roles. Here, for instance, we
mainly focus on non-degenerate contact forms. Then a suitable homology
theory associated with an action functional is utilized to detect
closed Reeb orbits. In our case, this is the equivariant symplectic
homology, i.e., essentially Floer theory.

The fundamental difficulty in the multiplicity problem, at least in
dimensions $2n+1>3$, lies not in the choice of homology theory but in
distinguishing simple orbits from iterated ones. This difficulty
already manifests itself in the classical problem of the existence of
infinitely many simple closed geodesics for a Riemannian metric on
$S^n$, which is wide open for $n>2$. To get around this problem, one
invariably has to impose restrictions on the index or action of closed
Reeb orbits.
 
To illustrate the state of the art of multiplicity results for
$2n+1>3$, let us consider the simplest example of the standard contact
structure $\xi$ on $S^{2n+1}$ without trying to give a comprehensive
account even in this case. Hypothetically, every contact form $\alpha$
supporting $\xi$ has at least $n+1$ simple closed Reeb orbits. This
conjecture, however, is very far from being proved when $n\geq
2$. (See \cite{CGH, GHHM, LL} for the proofs when $n=1$.) In general,
without any non-degeneracy or index/action assumptions, it is not even
known if there is more than one simple closed Reeb orbit if $n\geq
2$. When $\alpha$ is non-degenerate, it is easy to see that there must
be at least two such orbits (see, e.g., \cite{Gu, Ka}), but the
existence of three simple orbits on, say, $S^5$ is already a difficult
open question.

The situation changes dramatically once we impose further restrictions
on the indices or actions of closed Reeb orbits. Putting action
requirements aside, although these are also of considerable interest,
we will focus on the index constraints which are more relevant to our
goals here. In a series of papers starting with a groundbreaking work
of Long and Zhu, \cite{Lon02, LZ}, various multiplicity results have
been proved under what is usually referred to as the dynamical
convexity assumption; see \cite{AM:multiplicity, GG:convex, GuKa,
  Wa13, Wa} and references therein.  For $S^{2n+1}$ this is the
requirement that all closed Reeb orbits have Conley--Zehnder index at
least $n+2$ and follows from geometrical convexity; see, e.g.,
\cite{AM:dynconvex, GG:convex, HWZ:convex}.  (When the form is
degenerate, one has to replace the Conley--Zehnder index by its lower
semicontinuous extension.) Then it has been shown that a
non-degenerate dynamically convex contact form on $S^{2n+1}$ must have
at least $n+1$ simple closed Reeb orbits and, without the
non-degeneracy assumption, the number of orbits is at least
$\lceil (n+1)/2\rceil +1$.  Some of these results and methods carry
over to other contact manifolds, e.g., to certain prequantization
bundles, although then the notion of dynamical convexity gets more
involved; cf.\ \cite{AM:multiplicity}.

More recently, in \cite{DLLW}, the existence of $n+1$ simple closed
Reeb orbits for non-degenerate forms on $S^{2n+1}$ was established
under a condition which is less restrictive than dynamical convexity.
This condition is that all closed Reeb orbits have positive mean index
and there are no orbits with Conley--Zehnder index 0 when $n$ is odd
and index $0$ or $\pm 1$ when $n$ is even.

Our main goal in this paper is to extend this result to some other
prequantization bundles including the unit cotangent bundles of
CROSS's. This is done in Theorem \ref{thm:main} and its corollaries;
see Section \ref{sec:results}. In particular, we establish the
existence of at least two geometrically distinct closed geodesics for
a bumpy Finsler metric on a CROSS; cf.\ \cite{DLW1}. We also show that
many of the orbits found in the setting of Theorem \ref{thm:main} are
non-hyperbolic; see Section \ref{sec:non-hyp} and, in particular,
Theorem \ref{thm:non-hyp} and Corollary \ref{cor:sphere&cross-non-hyp}
generalizing the results from \cite{DLW2, DLLW}.  Finally, in Section
\ref{sec:deg}, for the sake of comparison we briefly touch upon the
case where the contact form is degenerate.

The proof of Theorem \ref{thm:main}, similarly to the proof of the
multiplicity theorem in \cite{DLLW}, hinges on a combinatorial result
-- the so-called enhanced common index jump theorem -- enabling one to
distinguish simple closed orbits from iterated ones. In Section
\ref{sec:IRT}, we revisit and reprove this theorem from the
perspective of index recurrence; cf.\ \cite[Sect.\ 5]{GG:convex}.

On the technical side, as has been mentioned above, the proof of our
main theorem relies on the machinery of equivariant symplectic
homology treated in a somewhat unconventional way following
\cite[Sect.\ 3]{GG:convex}; see Section \ref{sec:esh}.  This machinery
necessitates certain fillability requirements or index lower bounds,
which limit the class of prequantization bundles and contact forms in
the main theorem. If the equivariant symplectic homology is replaced
by contact homology, also used in a slightly non-standard form (see
Section \ref{sec:ch}), the main theorem can be further
generalized. This generalization is discussed in Section
\ref{sec:ch-m}; see Theorem~\ref{thm:main-ch}.

Another application of the variant of the contact homology from
Section \ref{sec:ch} is a refinement of the contact Conley conjecture
originally proved in \cite{GGM} and asserting the unconditional
existence of infinitely many simple closed Reeb orbits on some
prequantization bundles, not forced by homological growth. This is
Theorem \ref{thm:ccc} in Section \ref{sec:ccc}. We refer the reader to
\cite{GG:CC} for a detailed survey of the results on the Conley
conjecture.

%\subsection*{Acknowledgments}

\vskip .2cm
\noindent \textbf{Acknowledgments.}
The authors are grateful to Fr\'ed\'eric Bourgeois for useful discussions.

\section{Main results}
\label{sec:intro+results}

\subsection{Multiplicity results for closed Reeb orbits}
\label{sec:results}
Let $(M^{2n+1},\xi)$ be a closed contact manifold satisfying
$c_1(\xi)|_{\pi_2(M)}=0$ and let $\alpha$ be contact form supporting
the contact structure $\xi$. We call $\alpha$ \emph{index-positive}
(resp.\ \emph{index-negative}) if the mean index $\mi(\ga)$ is
positive (resp.\ negative) for every contractible periodic orbit $\ga$
of $\alpha$ and \emph{index-definite} when $\alpha$ is index-positive
or index-negative. Note that these requirements are notably weaker
than the standard notions of index positivity/negativity where, in,
say, the positive case, the mean index is required to grow at least
linearly with the action; cf.\ Lemma \ref{lemma:pos_index}. However,
the requirements become equivalent when the Reeb flow has only
finitely many contractible simple closed orbits.  The form $\alpha$ is
said to be \emph{index-admissible} if it has no closed orbits with
index $2-n$ or $2-n \pm 1$ contractible in $M$. Below, as is
customary, a non-degenerate periodic orbit $\ga$ is called \emph{good}
if its Conley-Zehnder index $\cz(\ga)$ has the same parity as that of
the underlying simple closed orbit; see Section \ref{sec:index_orbit}.

Throughout the paper we will focus on contact manifolds
$(M^{2n+1},\xi)$ which are prequantization circle bundles over closed
integral symplectic manifolds $(B^{2n},\omega)$, i.e., the first Chern
class of the principle bundle $M \to B$ is $-[\omega]$. We will
consider such contact manifolds which admit a ``nice'' symplectic
filling and also the non-fillable ones.  Accordingly, we will impose
one of the following two conditions, (F) and (NF), in most of our
results.

\begin{enumerate}%[style=multiline, labelwidth=.7cm]

\item[\reflb{cond:F}{(F)}] 
\begin{itemize}
\item[(i)] The manifold $(M^{2n+1},\xi)$ admits a strong symplectic
filling $(W,\Om)$ which is symplectically aspherical, i.e.,
$\Om|_{\pi_2(W)}=0$ and $c_1(TW)|_{\pi_2(W)}=0$, and the map
$\pi_1(M) \to \pi_1(W)$ induced by the inclusion is injective.

\item[(ii)] The contact form $\alpha$ is non-degenerate,
  index-definite and has no contractible good periodic orbits $\ga$
  such that $\cz(\ga)=0$ if $n$ is odd or $\cz(\ga) \in \{0,\pm 1\}$
  if $n$ is even.
\end{itemize}

\item[\reflb{cond:NF}{(NF)}] We have $c_1(\xi)=0$ in $H^2(M;\Z)$ and
  $B$ is spherically positive monotone. Furthermore, the contact form
  $\alpha$ is non-degenerate, index-positive, index-admissible and has
  no contractible good periodic orbits $\ga$ such that $\cz(\ga)=0$ if
  $n$ is odd or $\cz(\ga) \in \{0,\pm 1\}$ if $n$ is even.

\end{enumerate}

Note that in the setting of Part (i) of (F), $B$ is necessarily
spherically monotone. (We show this in the proof of Proposition
\ref{prop:CH}.)  Likewise, the condition that $c_1(\xi)=0$ from (NF)
implies via the Gysin exact sequence that $c_1(TB)=\lambda [\om]$ in
$H^2(B;\Q)$ for some $\lambda\in\R$, i.e., the symplectic manifold
$(B,\omega)$ is positive or negative monotone in a very strong
sense. (Then $\lambda\geq 0$ since $B$ is also spherically positive
monotone.)

A word is also due on the role of the condition that $\alpha$ is
simultaneously index-positive and index-admissible in (NF). This
condition is equivalent to that all contractible periodic orbits have
index greater than $3-n$ whenever the contact form is index-definite
(more precisely, whenever the contact form has no contractible closed
orbits with zero mean index). As a consequence, the positive
equivariant symplectic homology of $M$ is defined and well-defined
without a filling of $(M,\alpha)$ when (NF) holds; see Section
\ref{sec:esh} and \cite[Sect.\ 4.1.2]{BO:12}.

Our main result is Theorem \ref{thm:main} which, under some index
conditions, establishes a sharp lower bound for the number of
contractible closed Reeb orbits on certain prequantization
$S^1$-bundles. In what follows, given a symplectic
manifold $B$, denote by $\chi(B)$ its Euler characteristic and by
$$ 
c_B := \inf \{k \in \Z^+ \mid \exists S \in \pi_2(B) \text{ with }
\langle c_1(TB),S \rangle = k\} 
$$ 
its minimal Chern number.

\begin{theorem}
\label{thm:main}
Let $(M^{2n+1},\xi)$ be a prequantization $S^1$-bundle of a closed
symplectic manifold $(B,\om)$ such that $\om|_{\pi_2(B)}\neq 0$ and
$c_B>n/2$ and, furthermore, $H_{k}(B;\Q)=0$ for every odd $k$ or
$c_B>n$. Let $\alpha$ be a contact form supporting $\xi$ and assume
that $M$ and $\alpha$ satisfy condition \ref{cond:F} or \ref{cond:NF}.
Then $\alpha$ carries at least $r_B$ geometrically distinct
contractible periodic orbits, where
\begin{equation*}
r_B:=
\begin{cases}
\chi(B) + 2\dim H_n(B;\Q)
\text{ if } n\text{ is odd} \\
\chi(B) + 4\dim H_{n-1}(B;\Q)
\text{ if } n\text{ is even.}
\end{cases}
\end{equation*}
\end{theorem}
 
\begin{remark}
\label{rmk:lift}
Strictly speaking, the prequantization $(M^{2n+1},\xi)$ is uniquely
determined by a lift of the de Rham cohomology class of $\om$ to
$H^2(M;\Z)$ but not, in general, by the de Rham cohomology class
itself. The ambiguity in the lift is the torsion
$T=\mathrm{Tors}\big(H^2(B;\Z)\big)$, which by the universal
coefficient theorem is also equal to
$\mathrm{Tors}\big(H_1(B;\Z)\big)$; cf.\ \cite[Rmk.\ 2.3]{GGM}. In
what follows, we will tacitly assume that a lift is fixed and use the
notation $[\om]$ for either the lift or, depending on the context, the
de Rham cohomology class, i.e., an element of $H^2(M;\Z)/T$ or
$H^2(M;\Q)$. The notation $H^2(M;\Z)$ will always be used for the
entire integral cohomology group including the torsion, and the
condition $c_1(\xi)=0$, e.g., from \ref{cond:NF}, is understood as
vanishing of $c_1(\xi)$ in this group, not only modulo torsion.
\end{remark}

\begin{remark}
  The $r_B$ closed Reeb orbits from Theorem \ref{thm:main} need not be
  simple. These orbits can be iterates of simple non-contractible
  closed orbits and thus are simple only in the class of contractible
  orbits. Since the orbits are geometrically distinct, the theorem, in
  particular, implies the existence of $r_B$ simple orbits. The free
  homotopy classes of these orbits are necessarily torsion.
\end{remark}

The conditions of Theorem \ref{thm:main}, which are admittedly
somewhat technical and involved, can roughly speaking be divided into
three overlapping groups serving three different purposes and
deserving a further discussion.

We rely on equivariant symplectic homology in the proof of the
theorem, and the first group comprise the conditions needed to ensure
that this homology is defined and $\Z$-graded. Part (i) of
\ref{cond:F} is in this group.  In the non-fillable case,
\ref{cond:NF}, the condition that $\alpha$ is simultaneously
index-positive and index-admissible is equivalent to that all
contractible periodic orbits have index greater than $3-n$ whenever
the contact form is index-definite. As a consequence, the positive
equivariant symplectic homology of $M$ is defined and well-defined
without a filling of $M$; see Section \ref{sec:esh} and \cite[Sect.\
4.1.2]{BO:12}.  (However, if one uses the machinery of cylindrical
contact homology, weaker requirements would be sufficient; see Section
\ref{sec:sft}.)

Conditions from the second group are used to show that the positive
contractible equivariant symplectic homology is equal to the direct
sum of infinite number of copies of $H_*(B;\Q)$ with a certain degree
shift; see Proposition \ref{prop:CH}. Among these are, for instance,
the requirements that $\om|_{\pi_2(B)}\neq 0$ and that
$H_{\mathit{odd}}(B;\Q)=0$ or $c_B>n$, and also some parts of (NF).

Finally, the third group of conditions are employed to detect simple
closed Reeb orbits. These are the conditions that $c_B>n/2$ and that
$\alpha$ is index-definite and has no contractible good periodic
orbits $\ga$ such that $\cz(\ga)=0$ if $n$ is odd or
$\cz(\ga) \in \{0,\pm 1\}$ if $n$ is even.

\begin{remark}
  It is conceivable that the hypothesis $\Om|_{\pi_2(W)}=0$ in
  \ref{cond:F} can be dropped using Novikov rings. However, the
  condition $c_1(TW)|_{\pi_2(W)}=0$ in (F) or $c_1(\xi)|_{\pi_2(M)}=0$
  in \ref{cond:NF} seems crucial in our argument, since we need to use
  equivariant symplectic homology with an integer grading.
\end{remark}

At this stage we do not have any examples of $(M,\alpha)$ which would
satisfy the conditions of the theorem with
$H_{\mathit{odd}}(B;\Q)\neq 0$. In other words, in all the examples we
know $r_B=\dim H_*(B;\Q)$. However, we stated the theorem in this more
general form with an eye to possible generalizations and also to the
results from Section \ref{sec:sft}.

The following corollary follows immediately from Theorem
\ref{thm:main}.

\begin{corollary}
\label{cor:1}
Let $(M^{2n+1},\xi)$ be a prequantization $S^1$-bundle of a closed
symplectic manifold $(B,\om)$ such that $\om|_{\pi_2(B)}\neq 0$,
$c_B>n/2$ and $H_{k}(B;\Q)=0$ for every odd $k$. Let $\alpha$ be a
contact form supporting $\xi$. Assume that $M$ and $\alpha$ satisfy
either condition \ref{cond:F} or condition \ref{cond:NF}. Then
$\alpha$ carries at least $r_B$ geometrically distinct contractible
periodic orbits, where $r_B$ is the total rank of $H_\ast (B;\Q)$.
\end{corollary}

Examples satisfying the hypotheses of Corollary \ref{cor:1} include
the standard contact sphere $S^{2n+1}$ and the unit cosphere bundle of
a compact rank one symmetric space (CROSS). More precisely, $S^{2n+1}$
is the prequantization of $\CP^n$, and its obvious filling in
$\R^{2n+2}$ satisfies \ref{cond:F}. A compact rank one locally
symmetric space $N$ is a closed Riemannian manifold such that its
curvature tensor is invariant under parallel transport and the maximal
dimension of a flat totally geodesic submanifold is one. By the
classification of symmetric spaces, a CROSS must be
one of the following manifolds: $S^m$, $\RP^m$, $\CP^m$, $\HP^m$ and
$\CaP^2$; see \cite{Bes} for details. Thus the filling of the unit
cosphere bundle $S^*N$ given by the unit disk bundle $D^*N$ in $T^*N$
clearly meets the condition (F) unless $N$ is $S^2$ or $\RP^2$ (which
are the only cases where $\pi_1(S^*N) \to \pi_1(D^*N)$ is not
injective). However, in these cases it is well known that every Reeb
flow has at least two simple closed orbits.

Every CROSS $N$ admits a metric such that all of its geodesics are
periodic of the same minimal period; in other words, the geodesic flow
generates a free circle action on $S^*N$. Thus the unit cosphere
bundle $S^*N$ is the prequantization of a closed symplectic manifold
$(B,\om)$. Moreover, a homological computation shows that
$H_k(B;\Q)=0$ for every odd $k$; see \cite[page 141]{Zil}.  In this
case, the total rank $r_B$ of $H_*(B;\Q)$ and the minimal Chern number
$c_B$ are given in the following table.
\[
\begin{tabular}{ | c | c | c | }
%\tabulinesep=1.2mm
\hline
Prequantization & $r_B = \dim H_*(B;\Q)$ & $c_B$ \\
\hline
\rule{0pt}{0.4cm}
$S^{2n+1}$ & $n+1$ & $n+1$ \\
$S^*S^{2}$ or $S^*\RP^2$ & $2$ & $2$ \\
$S^*S^{m}$ or $S^*\RP^m$ \text{with} $m>2$ \text{even} & $m$ & $m-1$ \\
$S^*S^{m}$ or $S^*\RP^m$ \text{with} $m$ \text{odd} & $m+1$ & $m-1$ \\
$S^*\CP^{m}$ & $m(m+1)$ & $m$ \\
$S^*\HP^m$ & $2m(m+1)$ & $2m+1$ \\
$S^*\CaP^2$ & $24$ & $11$ \\
\hline
\end{tabular}
\]

Notice that the hypothesis on $c_B$ in Corollary \ref{cor:1} barely
holds for $M=S^*\CP^m$, where $\dim B/4=m-1/2$ and $c_B=m$.  We have
the following consequence of Corollary \ref{cor:1}, which was
previously proved for the standard contact sphere by Duan, Liu, Long
and Wang in \cite{DLLW} and for Finsler metrics on a simply connected
CROSS by Duan, Long and Wang in \cite{DLW2}.

\begin{corollary}
\label{cor:sphere&cross}
Let $(M,\xi)$ be either the standard contact sphere $S^{2n+1}$ or the
unit cosphere bundle $S^*N$ of a CROSS and let $\alpha$ be a contact
form supporting $\xi$. Assume that $\alpha$ satisfies condition
\ref{cond:F}. Then $\alpha$ has at least $r_B$ geometrically distinct
periodic orbits, where $r_B$ is given by the table above.
\end{corollary}

The standard contact sphere and the unit cosphere bundle of a CROSS
(with dimension bigger than two) satisfy the assumption \ref{cond:F}
and therefore the only condition on the contact form in Corollary
\ref{cor:sphere&cross} is that it is index-definite and has no good
contractible periodic orbits $\ga$ such that $\cz(\ga)=0$ if $n$ is
odd or $\cz(\ga) \in \{0,\pm 1\}$ if $n$ is even.  Furthermore, the
prequantization bundles in Corollary \ref{cor:sphere&cross} admit
contact forms with precisely $r_B$ geometrically distinct periodic
orbits. These contact forms are given by irrational ellipsoids and the
Katok-Ziller Finsler metrics; \cite{Zil}. This shows that the lower
bound in Theorem \ref{thm:main} is sharp. To the best of our
knowledge, all the examples of prequantization $S^1$-bundles admitting
contact forms with finitely many simple closed Reeb orbits known so
far satisfy the hypothesis that $H_*(B;\Q)$ vanishes in odd degrees.

As an easy application of Theorem \ref{thm:main}, we establish, with
no index assumptions, the existence of at least two geometrically
distinct contractible closed orbits for \emph{any} non-degenerate
contact form on manifolds as in Corollary \ref{cor:1} satisfying
\ref{cond:F}; see Section \ref{sec:proof cor2}.

\begin{theorem}
\label{thm:app}
Let $(M^{2n+1},\xi)$ be a prequantization $S^1$-bundle of a closed
symplectic manifold $(B,\om)$ such that $\om|_{\pi_2(B)}\neq 0$,
$c_B>n/2$, and $H_{k}(B;\Q)=0$ for every odd $k$. Assume that $M$
satisfies Part (i) of condition \ref{cond:F}. Then every
non-degenerate contact form $\alpha$ supporting $\xi$ has at least two
geometrically distinct contractible periodic orbits.
\end{theorem}

This theorem combined with the above discussion implies the following
corollary.

\begin{corollary}
  Let $(M,\xi)$ be the standard contact sphere $S^{2n+1}$ or the unit
  cosphere bundle $S^*N$ of a CROSS. Then every non-degenerate contact
  form supporting $\xi$ carries at least two geometrically distinct
  closed orbits.
\end{corollary}

The next result is closely related to a theorem of Duan, Long and Wang
asserting the existence of two geometrically distinct closed geodesics
for a bumpy metric on a simply connected manifold; \cite{DLW1}.

\begin{corollary}
  Every bumpy Finsler metric on a CROSS has at least two geometrically
  distinct closed geodesics.
\end{corollary}

There are also a few examples where $M$ does not obviously meet the
requirements of Part (i) of \ref{cond:F} but for a suitable form
$\alpha$ can satisfy \ref{cond:NF}. Among these are the
prequantizations of the following manifolds $(B,\omega)$: the complex
Grassmannians $\Gr_{\C}(2; m)$, $\Gr_{\C}(3; 6)$ and $\Gr_{\C}(3; 7)$,
the monotone products $\CP^m\times \CP^m$ (cf., \cite[Sect.\
1.2]{GG:hyp}) and also the monotone products
$\CP^m \times \Gr_{\R}^+(2; m+3)$ where the second factor is a real
oriented Grassmannian and its minimal Chern number is $m+1$.  For
these manifolds $B$ the lower bound $r_B$ from Theorem \ref{thm:main}
is sharp. (The reason is that $B$ admits a Hamiltonian circle or torus
action with isolated fixed points. Such an action has exactly $r_B$,
the sum of Betti numbers, fixed points and the required Reeb flow is
then obtained by lifting a flow generating the action to $M$.)

\subsection{Existence of non-hyperbolic periodic orbits}
\label{sec:non-hyp}
The proof of Theorem \ref{thm:main} also yields the following
multiplicity result concerning non-hyperbolic closed orbits when the
contact form has finitely many geometrically distinct contractible
closed orbits. Recall that a closed orbit is hyperbolic if its
linearized Poincar\'e map has no eigenvalues on the unit circle.

\begin{theorem}
\label{thm:non-hyp}
Let $(M^{2n+1},\xi)$ be a prequantization $S^1$-bundle of a closed
symplectic manifold $(B,\om)$ such that $\om|_{\pi_2(B)}\neq 0$ and
$c_B>n/2$ and, furthermore, $H_{k}(B;\Q)=0$ for every odd $k$ or
$c_B>n$. Let $\alpha$ be a contact form supporting $\xi$ with finitely
many geometrically distinct contractible closed orbits. Assume that
$M$ and $\alpha$ satisfy either condition \ref{cond:F} or condition
\ref{cond:NF}. Then $\alpha$ carries at least $r^\text{non-hyp}_B$
geometrically distinct contractible non-hyperbolic periodic orbits,
where
\begin{equation*}
r^\text{non-hyp}_B:= r_B - \dim H_n(B;\Q) =
\begin{cases}
  \chi(B) + \dim H_n(B;\Q)
\text{ if } n\text{ is odd} \\
  \chi(B) + 4\dim H_{n-1}(B;\Q) 
- \dim H_n(B;\Q)\text{ if } n\text{ is
    even.}
\end{cases}
\end{equation*}
\end{theorem}

This result immediately follows from the proof of Theorem
\ref{thm:main}; see Remarks \ref{rmk:non-hyp1} and
\ref{rmk:non-hyp2}. Clearly, under the additional assumption that the
contact form has finitely many geometrically distinct contractible
closed orbits, all the applications of Theorem \ref{thm:main} have
analogous statements replacing $r_B$ by $r^\text{non-hyp}_B$. For
instance, when $M$ is the standard contact sphere or the unit cosphere
bundle of a CROSS, a computation yields the following table:
\[
\begin{tabular}{ | c | c | }
\hline
Prequantization & $r^\text{non-hyp}_B$ \\
\hline
\rule{0pt}{0.4cm}
$S^{2n+1}$ \text{with} n \text{even} & $n$ \\
$S^{2n+1}$ \text{with} n \text{odd} & $n+1$ \\
$S^*S^{m}$ or $S^*\RP^m$ \text{with} m \text{even} & $m$ \\
$S^*S^{m}$ or $S^*\RP^m$ \text{with} m \text{odd} & $m-1$ \\
$S^*\CP^{m}$ & $m(m+1)$  \\
$S^*\HP^m$ & $2m(m+1)$ \\
$S^*\CaP^2$ & $24$ \\
\hline
\end{tabular}
\]
Thus we obtain the following corollary which, again, was previously
proved for the standard contact sphere in \cite{DLLW} and for Finsler
metrics on a simply connected CROSS in \cite{DLW2}.

\begin{corollary}
\label{cor:sphere&cross-non-hyp}
Let $(M,\xi)$ be the standard contact sphere $S^{2n+1}$ or the unit
cosphere bundle $S^*N$ of a CROSS. Let $\alpha$ be a contact form
supporting $\xi$ satisfying the hypothesis \ref{cond:F} and having
finitely many geometrically distinct contractible closed orbits. Then
$\alpha$ has at least $r^\text{non-hyp}_B$ non-hyperbolic
geometrically distinct contractible periodic orbits, where $r^\text{non-hyp}_B$ is
given by the previous table.
\end{corollary}
 
\subsection{The case of a degenerate form}
\label{sec:deg}
It is interesting to compare Theorem \ref{thm:main} with the lower
bounds one has without the non-degeneracy condition on the form
$\alpha$. In this case, the index restrictions become much more severe
and the lower bound $r$ on the number of simple closed Reeb orbits
much weaker. In particular, $r$ depends only on the dimension of $M$
and the index lower bound but not on the topology of $B$. To be more
precise, denote by $\mu_-$ the lower semicontinuous extension of the
Conley--Zehnder index; see, e.g., \cite[Sect. 3]{AM:dynconvex} or
\cite[Sect.\ 4.1.2]{GG:convex}. We have the following result:

\begin{theorem}
\label{thm:deg}
Let $(M^{2n+1},\xi)$ be a prequantization $S^1$-bundle of a closed
symplectic manifold $(B,\om)$ such that $\om|_{\pi_2(B)}\neq 0$ and
$c_B>n/2$ and, furthermore, $H_{k}(B;\Q)=0$ for every odd $k$ or
$c_B>n$. Assume, in addition, that $M$ satisfies Part (i) of condition
\ref{cond:F} and the filling $W$ is exact.  Let $\alpha$ be a contact
form supporting $\xi$ such that $\mu_-(\gamma)\geq q$ for all, not
necessarily simple, contractible closed Reeb orbits $\gamma$.  Then
$M$ carries at least $r$ geometrically distinct contractible closed
Reeb orbits, where
$$
r=
\begin{cases}
q-\lceil (n+1)/2 \rceil &\textrm{ when $n$ is even and $q$ is odd,}\\
q+1-\lceil (n+1)/2 \rceil & \textrm{otherwise.}
\end{cases}
$$
\end{theorem}

Here the result is void if $r\leq 0$. The main class of manifolds this
theorem applies to is again the unit cotangent bundles of CROSS's. For
$S^*S^m$ (already considered in \cite{GG:convex}) and $S^*\RP^m$ the
theorem yields, depending on $q$, the existence of a number of
geometrically distinct periodic orbits and of two such orbits for
$S^*\HP^1$ when $q=3$. The main limitation comes from the fact that
$q$ cannot be larger than the minimal degree $d$ where the relevant
symplectic homology for contractible orbits is non-trivial. For
$S^*S^m$ and $S^*\RP^m$ (with $m>2$), we have $d= m-1$; for
$S^*\CP^m$, $S^*\HP^m$ and $S^*\CaP^2$, we have $d=1$, $3$ and,
respectively, $7$; see \cite{AM:dynconvex}.  Most likely, Theorem
\ref{thm:deg}, in contrast with Theorem \ref{thm:main}, is very far
from being sharp. In fact, one can expect that a degenerate form
necessarily has infinitely many simple closed Reeb orbits and, in
particular, Theorem \ref{thm:main} holds without any non-degeneracy
assumptions.

The proof of Theorem \ref{thm:deg} uses Lusternik--Schnirelmann theory
for the shift operator in equivariant symplectic homology developed in
\cite{GG:convex} and a variant of the index recurrence theorem for a
degenerate paths from \cite[Sect.\ 5]{GG:convex} or the common jump
theorem from \cite{Lon02,LZ}. The argument is essentially identical to
the proofs of \cite[Thm.\ 6.9 and Thm.\ 6.15]{GG:convex} and we omit
it. The requirement that $W$ is exact is needed to ensure that the
Hamiltonian action filtration of the symplectic homology agrees with
the contact action.

\section{Preliminaries}

In this section we will review some basic concepts from the
Conley-Zehnder index theory and equivariant symplectic homology used
throughout the paper.

\subsection{The Conley--Zehnder index for paths of symplectic matrices}
\label{sec:index_paths}
To every continuous path $\Phi\colon [0,\,1]\to\Sp(2n)$ beginning at
$\Phi(0)=I$, one can associate the mean index $\hmu(\Phi)\in \R$, a
homotopy invariant of the path with fixed end-points.  The mean index
$\hmu(\Phi)$ measures the total rotation angle of certain unit
eigenvalues of $\Phi(t)$ and $\hmu(\Phi_s)=\const$ for a family of
paths $\Phi_s$ as long as the eigenvalues of $\Phi_s(1)$ remain
constant. The resulting map $\hmu\colon \widetilde{\Sp}(2n)\to\R$ is a
unique quasimorphism on the universal covering $\TSp(2n)$ of $\Sp(2n)$
which is continuous and homogeneous, i.e.,
$$
\hmu(\Phi^k)=k\hmu(\Phi),
$$
and satisfies the normalization condition
$$
\hmu(\Phi_0)=2\quad \textrm{for}\quad \Phi_0(t)= \exp\big(2\pi \sqrt{-1}
t\big)\oplus I_{2n-2}
$$ 
with $t\in [0,\,1]$; see \cite{BG}. The quasimorphism condition
asserts that $\hmu$ fails to be a homomorphism only up to a constant,
i.e.,
\begin{equation}
\label{eq:qm}
\big|\hmu(\Phi\Psi)-\hmu(\Phi)-\hmu(\Psi)\big|\leq C_n,
\end{equation} 
where the constant is independent of $\Phi$ and $\Psi$, but may depend
on $n$. (In fact, one may be able to take $C_n=4n$; \cite{Us}.)  We
refer the reader to \cite{Lon02,SZ} for a very detailed discussion of
the mean index. In this paper we use conventions and notation from
\cite[Sec.\ 4]{GG:convex}.

Assume next that the path $\Phi$ is non-degenerate, i.e., by
definition, all eigenvalues of the end-point $A=\Phi(1)$ are different
from one. We denote the set of such matrices $A\in\Sp(2n)$ by
$\Sp^*(2n)$ and also denote the part of $\TSp(2n)$ lying over
$\Sp^*(2n)$ by $\TSpn(2n)$. It is not hard to see that $A$ can be
connected to a symplectic transformation with elliptic part equal to
$-I$ (if non-trivial) by a path $\Psi$ lying entirely in
$\Sp^*(2n)$. Concatenating this path with $\Phi$, we obtain a new path
$\Phi'$. By definition, the Conley--Zehnder index $\mu(\Phi)\in\Z$ of
$\Phi$ is $\hmu(\Phi')$. One can show that $\mu(\Phi)$ is
well-defined, i.e., independent of $\Psi$. The function
$\mu\colon \TSpn(2n)\to\Z$ is locally constant, i.e., constant on
connected components of $\TSpn(2n)$. In other words,
$\mu(\Phi_s)=\const$ for a family of paths $\Phi_s$ as long as
$\Phi_s(1)\in\Sp^*(2n)$ for every $s$.  Furthermore, we call $\Phi$
strongly non-degenerate if all its ``iterations'' $\Phi^k$ are
non-degenerate, i.e., none of the eigenvalues of $\Phi(1)$ is a root
of unity.

In the rest of this section we briefly discuss the properties of the
Conley--Zehnder type indices which are essential for our purposes,
referring the reader to, e.g., \cite{Lon02,SZ} for the proofs. Below
all paths are required to begin at $I$ and are taken up to homotopy,
i.e., as elements of $\TSp(2n)$.

We start with three specific examples. For the path
$\Phi(t)=\exp\big(2\pi\sqrt{-1}\lambda t\big)$, $t\in [0,\,1]$, we
have
$$
\hmu(\Phi)=2\lambda\textrm{ and } 
\mu(\Phi)=\sign(\lambda)\big(2\lfloor|\lambda|\rfloor +1\big)
\textrm{ when $\lambda\not\in\Z$. }
$$
Next, let $H$ be a non-degenerate quadratic form on $\R^{2n}$ with
eigenvalues in the range $(-\pi,\,\pi)$. (The eigenvalues of a
quadratic form $H$ on a symplectic vector space are by definition the
eigenvalues of its Hamiltonian vector field $X_H=J \nabla H$, where
$J$ is the matrix of the symplectic form.) The path
$\Phi(t)=\exp(JHt)$, $t\in [0,\,1]$, is the linear autonomous
Hamiltonian flow generated by $H$. Then, with our conventions,
$$
\mu(\Phi)=\frac{1}{2}\sgn(H),
$$
where $\sgn(H)$ is the signature of $H$, i.e., the number of positive
squares minus the number of negative squares in the diagonal form of
$H$ with $\pm 1$ and $0$ on the diagonal. In addition, when $\Phi(1)$
is hyperbolic, we have
$$
\mu(\Phi)=\hmu(\Phi).
$$
Furthermore,
$$
\mu(\Phi^{-1})=-\mu(\Phi)
$$
for any non-degenerate path $\Phi$. When $\varphi$ is a loop, we also
have
$$
\mu(\varphi\Phi)=\hmu(\varphi)+\mu(\Phi).
$$
Finally, $\hmu$ and $\mu$ are additive under direct sum. Namely, for
$\Phi\in\TSp(2n)$ and $\Psi\in\TSp(2n')$, we have
$$
\hmu(\Phi\oplus\Psi)=\hmu(\Phi)+\hmu(\Psi)
\quad\textrm{and}\quad
\mu(\Phi\oplus\Psi)=\mu(\Phi)+\mu(\Psi),
$$
where in the second identity we assumed that both paths are
non-degenerate. The mean index and the Conley--Zehnder index are
related by the inequality
$$
|\hmu(\Phi)-\mu(\Phi)|<n
$$
where $\Phi\in\TSp^*(2n)$. As a consequence,
$$
\lim_{k\to\infty}\frac{\mu(\Phi^k)}{k}=\hmu(\Phi),
$$
and hence the name ``mean index'' for $\hmu$.

\subsection{The Conley--Zehnder index of periodic orbits}
\label{sec:index_orbit}

Let $\ga$ be a strongly non-degenerate periodic orbit of the Reeb
vector field $R_\alpha$ and
$\Psi\colon \ga^*\xi \to S^1 \times \R^{2n}$ a symplectic
trivialization of $\xi$ over $\ga$. Denote by
$\Psi_t\colon \xi(\ga(t)) \to \R^{2n}$ the composition of
$\Psi|_{\ga^*\xi(t)}$ with the projection onto the second factor. Via
this trivialization, the linearized Reeb flow gives rise to the
symplectic path
\[
  \Phi(t) = \Psi_t \circ d\phi_\alpha^t(\gamma(0))|_\xi \circ
  \Psi_0^{-1},
\]
where $\phi^t_\alpha$ is the Reeb flow of $\alpha$. In this way, we
define the Conley--Zehnder index and the mean index of $\gamma$ with
respect to the trivialization $\Psi$ as
\[
\cz(\ga;\Psi) = \cz(\Phi)\text{ and }\mi(\ga;\Psi) = \mi(\Phi)
\]
respectively. The Conley--Zehnder index and the mean index depend only
on the homotopy class of $\Psi$. Indeed, if we choose another
trivialization $\Upsilon\colon \ga^*\xi \to S^1 \times \R^{2n}$ then
we have the relation
\[
  \cz(\ga;\Upsilon) = \cz(\ga;\Phi) + 2\maslov(\Upsilon_t \circ
  \Phi_t^{-1}),
\]
where $\maslov$ denotes the Maslov index which is a suitably chosen
one of the two isomorphisms between $\pi_1(\Sp(2n))$ and $\Z$.  In
particular, the parity of the index does not depend on the choice of
the trivialization. It turns out that the parities of the
Conley--Zehnder indices of the even/odd iterates of a periodic orbit
are the same, i.e., for all $j,\,k \in \N$,
\[
  \cz(\ga^{2j};\Psi^{2j}) \equiv \cz(\ga^{2k};\Psi^{2k})\ \text{and}\
  \cz(\ga^{2j-1};\Psi^{2j-1}) \equiv \cz(\ga^{2k-1};\Psi^{2k-1})\
\pmod{2}. 
\]

A periodic orbit of $\alpha$ is called good if its Conley--Zehnder
index has the same parity as that of the index of the underlying
simple closed orbit. (As has just been pointed out, the parity of the
index does not depend on the choice of the trivialization of $\xi$.) A
periodic orbit that is not good is called bad.

If $\ga$ is contractible, there is a standard way to choose the
trivialization $\Psi$ unique up to homotopy. Namely, consider a
capping disk of $\ga$, i.e., a smooth map $\sigma\colon D^2 \to M$,
where $D^2$ is the two-dimensional disk, such that
$\sigma|_{\partial D^2} = \ga$. Choose a trivialization of
$\sigma^*\xi$ and let $\Psi\colon \ga^*\xi \to S^1 \times \R^{2n}$ be
its restriction to the boundary, which gives a trivialization of $\xi$
over $\ga$. Since $D^2$ is contractible, the homotopy class of $\Psi$
does not depend on the choice of the trivialization of
$\sigma^*\xi$. Moreover, the condition that $c_1(\xi)|_{\pi_2(M)}=0$
ensures that the homotopy class of $\Psi$ does not depend on the
choice of $\sigma$ as well. Throughout the paper, whenever $\ga$ is
contractible, we denote by $\mu(\ga)$ and $\mi(\ga)$ the
Conley--Zehnder index and, respectively, the mean index of $\ga$ with
respect to the standard trivialization.

\subsection{Equivariant symplectic homology}
\label{sec:esh}
In this section we briefly recall several facts about positive
equivariant symplectic homology, treating the subject from a slightly
unconventional perspective.

Let first $(M,\xi)$ be a closed contact manifold and $(W,\Omega)$ be a
strong symplectic filling of $M$ with
$\Om|_{\pi_2(W)}=0=c_1(TW)|_{\pi_2(W)}$. Furthermore, let $\alpha$ be
a non-degenerate contact form on $M$ supporting the contact structure
$\xi$.  Then the positive equivariant symplectic homology
$\SH^{S^1,+}(W)$ with coefficients in $\Q$ is the homology of a
complex $\CC_*(\alpha)$ generated by the good closed Reeb orbits of
$\alpha$; see \cite[Prop.\ 3.3]{GG:convex}. This complex is graded by
the Conley--Zehnder index and filtered by the action. Furthermore,
once we fix a free homotopy class of loops in $W$, the part of
$\CC_*(\alpha)$ generated by closed Reeb orbits in that class is a
subcomplex. As a consequence, the entire complex $\CC_*(\alpha)$
breaks down into a direct sum of such subcomplexes indexed by free
homotopy classes of loops in $W$.

The differential in the complex $\CC_*(\alpha)$, but not its homology,
depends on several auxiliary choices, and the nature of the
differential is not essential for our purposes. The complex
$\CC_*(\alpha)$ is functorial in $\alpha$ in the sense that a
symplectic cobordism equipped with a suitable extra structure gives
rise to a map of complexes. For the sake of brevity and to emphasize
the obvious analogy with contact homology, we denote the homology of
$\CC_*(\alpha)$ by $\HC_*(M)$ rather than $\SH^{S^1,+}(W)$. The
homology of the subcomplex formed by the orbits contractible in $W$
will be denoted by $\HC_*^0(M)$. However, it is worth keeping in mind
that $\CC_*(\alpha)$ and hypothetically even the homology may depend
on the choice of the filling $W$.

This description of the positive equivariant symplectic homology as
the homology of $\CC_*(\alpha)$ is not quite standard, but it is most
suitable for our purposes. (We refer the reader to \cite{GG:convex}
for more details and further references and to \cite{BO:Gysin,Vi} for
the original construction of the equivariant symplectic homology.) To
see why $\HC_*(M):=\SH^{S^1,+}(W)$ can be obtained as the homology of
a single complex generated by good closed Reeb orbits, let us first
consider an admissible Hamiltonian $H$ on the symplectic completion of
$W$ and focus on the orbits of $H$ with positive action. Such orbits
are in a one-to-one correspondence with closed Reeb orbits $\gamma$
with action below a certain threshold $T$ depending on the slope of
$H$. The $S^1$-equivariant Floer homology of $H$ is the homology of a
Floer-type complex obtained from a non-degenerate parametrized
perturbation of $H$; \cite{BO:Gysin,Vi}. This complex is filtered by
the action. The $E^1$-term of the resulting spectral sequence (over
$\Q$) is generated by the good Reeb orbits of $\alpha$ with action
below $T$. Now we can (canonically, once the generators are fixed)
reassemble the differentials $\p_r$ into a single differential $\p$ on
$\CC_*(H):=E^1_{*,*}$ in such a way the the homology of the resulting
complex is $E^\infty=\HF_*^{S^1,+}(H)$. Roughly speaking,
$\p=\p_1+\p_2+\ldots$, where $\p_r$ is suitably ``extended'' from
$E^r$ to $E^1$. Moreover, this procedure respects the action
filtration and is functorial with respect to continuation maps.
Passing to the limit in $H$, we obtain the complex $\CC_*(\alpha)$ as
the limit of the complexes $\CC_*(H)$; see \cite[Sect.\ 2.5 and
3]{GG:convex} for further details.

A remarkable observation by Bourgeois and Oancea in \cite[Sect.\
4.1.2]{BO:12} is that under suitable additional assumptions on the
indices of closed Reeb orbits the positive equivariant symplectic
homology is defined and well-defined even when $M$ does not have a
symplectic filling. To be more precise, assume that
$c_1(\xi)|_{\pi_2(M)}=0$ and let $\alpha$ be a non-degenerate contact
form on $M$ such that all of its closed contractible Reeb orbits have
Conley--Zehnder index strictly greater than $3-n$. Furthermore, under
this assumption the proof of \cite[Prop.\ 3.3]{GG:convex} carries over
essentially word-for-word, and hence again the positive equivariant
symplectic homology of $M$ can be described as the homology of a
complex $\CC_*(\alpha)$ generated by good closed Reeb orbits of
$\alpha$, graded by the Conley--Zehnder index and filtered by the
action. The complex breaks down into the direct sum of subcomplexes
indexed by free homotopy classes of loops in $M$. As in the fillable
case, we will use the notation $\HC_*(M)$ and $\HC_*^0(M)$.

The assumption that all contractible orbits have index greater than
$3-n$ is equivalent to that $\alpha$ is simultaneously index-positive
and index-admissible (assuming that there is no contractible closed
orbit with zero mean index), which are parts of the requirement
\ref{cond:NF}. Indeed, index positivity implies that all contractible
orbits have index greater than $-n$ and the condition that $\alpha$ is
index-admissible rules out the orbits of index $1-n$, $2-n$ and
$3-n$. (The converse is obvious if there is no contractible periodic
orbit with zero mean index.)  Hence in case \ref{cond:NF} of Theorem
\ref{thm:main} the positive equivariant symplectic homology of $M$ is
defined and well-defined without a filling of $M$.

\subsection{Equivariant symplectic homology of prequantizations}
\label{sec:esh_preq}

The next proposition shows how to compute the equivariant symplectic
homology of a suitable prequantization in terms of the homology of the
basis. This computation will be crucial throughout this work.

\begin{proposition}
\label{prop:CH}
Let $(M^{2n+1},\xi)$ be a prequantization of a closed symplectic
manifold $(B,\om)$ with $\omega|_{\pi_2(B)} \neq 0$ and  such that
$H_{k}(B;\Q)=0$ for every odd $k$ or $c_B>n$. 
\begin{itemize}
\item[{\rm (a)}] Assume that $M$ satisfies the requirements from Part
  (i) of \ref{cond:F}. Then, $B$ is spherically monotone.  When $B$ is
  spherically positive monotone, the positive equivariant symplectic
  homology for contractible periodic orbits of $M$ is given by
\begin{equation}
\label{eq:CH}
\HC^0_\ast(M) \cong \bigoplus_{m\in\N} H_{\ast -2mc_B+n} (B; \Q). 
\end{equation}
When $B$ is spherically negative monotone, we have
\begin{equation}
\label{eq:CH-nm}
\HC^0_\ast(M) \cong \bigoplus_{m\in\N} H_{\ast +2mc_B-n} (B; \Q). 
\end{equation}
In particular, in both cases the homology is independent of the choice
of the filling $W$ satisfying Part (i) of \ref{cond:F}.

\item[{\rm (b)}] Alternatively, assume that $B$ is spherically
  positive monotone with $c_B\geq 2$ and, as in \ref{cond:NF},
  $c_1(\xi)=0$ and $\alpha$ is a non-degenerate contact form on
  $(M,\xi)$ such that all closed Reeb orbits have index greater than
  $3-n$. Then \eqref{eq:CH} also holds.
\end{itemize}
\end{proposition}
In other words, \eqref{eq:CH} asserts that $\HC^0_\ast(M)$ is obtained
by taking an infinite number of copies of $H_{\ast-n}(B;\Q)$ with
grading shifted up by positive integer multiples of $2c_B$ and adding
up the resulting spaces.

\begin{remark}
\label{rmk:monotone}
Note that while the only known spherically positive monotone manifold
meeting the requirements $c_B>n$ is $\CP^n$, there are numerous
negative monotone manifolds satisfying this condition, e.g., complete
intersections of high degree. Also recall that in (b), we necessarily
have $c_1(TB)=\lambda [\om]$ in $H^2(B;\Q)$ for some $\lambda\in\R$,
i.e., the symplectic manifold $(B,\omega)$ is positive or negative
monotone in a very strong sense. (Then $\lambda\geq 0$ since $B$ is
also spherically positive monotone.) This follows from the condition
that $c_1(\xi)=0$ and the Gysin exact sequence.
\end{remark}

It is worth pointing out that in Case (a) of Proposition \ref{prop:CH}
the conditions, although quite restrictive, are purely of topological
nature and ultimately imposed only on the symplectic manifold
$(B,\omega)$. The homology in this case is defined for any contact
form and given by \eqref{eq:CH} or \eqref{eq:CH-nm}. On the other
hand, in Case (b) the conditions are imposed on both the manifold and
the contact form $\alpha$ and the homology is defined and satisfies
\eqref{eq:CH} only when $\alpha$ meets those requirements. Finally,
note that the requirement that $c_B\geq 2$ from (b) is automatically
satisfied in the setting of Theorem \ref{thm:main} as a consequence of
the assumptions $\omega|_{\pi_2(B)}\neq 0$ and $c_B>n/2$. Indeed, then
$c_B\geq 2$ when $n>1$ and for $n=1$ we necessarily have $B=S^2$ and
hence $c_B=2$.

\begin{proof}[Proof of Proposition \ref{prop:CH}]
  Let us focus first on Case (a). To show that $B$ is spherically
  monotone note that $TW|_M$ decomposes as the direct sum of $\xi$ and
  a trivial complex line bundle. Hence $c_1(\xi)$ is the image of
  $c_1(TW)$ in $H^2(M;\Z)$ and, as a consequence, $c_1(\xi)$ is an
  aspherical class. Next, arguing by contradiction, assume that $B$ is
  not spherically monotone. Then $c_1(TB)$ and $[\omega]$ are linearly
  independent as maps from $\pi_2(B)\otimes \Q$ to $\Q$. Therefore, as
  is easy to see, there exists $S\in\pi_2(M)$ such that
  $\left<c_1(TB), S\right>>0$ but $\left<[\om], S\right>=0$. The
  restriction of the prequantization bundle to $S$ is trivial and $S$
  admits a lift $S'$ to $M$. Then, since $c_1(\xi)$ is the pull-back
  of $c_1(TB)$, we have
$$
\left<c_1(\xi), S'\right>=\left<c_1(TB), S\right>>0.
$$ 
This is impossible because $c_1(\xi)$ is aspherical.

For the sake of simplicity we will assume throughout the rest of the
proof of Case (a) that $B$ is positive monotone. (When $B$ is negative
monotone, the argument is similar up to some sign changes.) Then, as
has been pointed out above, the positive equivariant symplectic
homology is defined and well-defined for any contact form supporting
$\xi$. Let $\alpha_0$ be a connection contact form on $(M,\xi)$. This
form is not non-degenerate, but rather Morse-Bott non-degenerate. Let
$a>0$ be the rationality constant of $(B,\om)$, i.e., the positive
generator of $\left<\om,\pi_2(B)\right>\subset \R$. Then the action
spectrum of $\alpha_0$ is $a\N$. Pick small non-overlapping intervals
$I_m=[ma-\ep,ma+\ep]$ with $\ep>0$.

A standard Morse--Bott type argument shows that  
$$
\HC^{I_m,0}_\ast(\alpha_0) \cong H_{\ast -2mc_B+n} (B; \Q),
$$
where on the left we have the filtered homology of $\alpha_0$ or to be
more precise of a small non-degenerate perturbation $\alpha$ of
$\alpha_0$; cf., e.g., \cite{Poz} and also \cite{Bo:thesis, BO:09} for
a different approach. Furthermore, the contractible positive
equivariant symplectic homology of $\alpha$ can be viewed as the
homology of a certain complex generated by good closed Reeb orbits;
see Section \ref{sec:esh}. This complex is filtered by action, and the
$E^1$-page of the resulting Morse--Bott spectral sequence is given by
the right-hand side of \eqref{eq:CH}. Namely,
$$
E^1_{m,q}= \HC^{I_m,0}_{m+q}(\alpha_0) \cong H_{m+q-2mc_B+n} (B; \Q).
$$
The condition that $H_{\mathit{odd}}(B;\Q)=0$ or $c_B>n$ readily
implies that this spectral sequence collapses in the $E^1$-term:
$E^1=E^\infty=\HC^0_\ast(M)$, which proves~\eqref{eq:CH}.

This argument applies in Case (b) word-for-word with one
nuance. Namely, to carry out the Morse--Bott calculation for
$\alpha_0$ we need to make sure that it admits an arbitrarily small
non-degenerate perturbation $\alpha$ such that all good closed Reeb
orbits of $\alpha$ have Conley--Zehnder index greater than $3-n$. This
is a consequence of the following lemma.

\begin{lemma}
\label{lemma:pos_index}
Let $(M,\xi)$ be the prequantization $S^1$-bundle over $(B,\omega)$
with connection contact form $\alpha_0$ such that 
$\xi=\ker\alpha_0$. Assume that $B$ is spherically positive monotone,
$\omega|_{\pi_2(B)}\neq 0$, and $c_1(\xi)=0$.  Let $\beta$ be a
sufficiently small non-degenerate perturbation of $\alpha_0$.  Then
$\mu(\gamma)\geq 2c_B-n$ for every contractible closed Reeb orbit of
$\beta$. Furthermore, there exists a constant $\Delta>0$, independent
of $\beta$, such that for all $\gamma$ we have
$$
\hmu(\gamma)\geq \Delta \cdot T(\gamma),
$$
where $T(\gamma)$ is the period (i.e., the action) of $\gamma$.
\end{lemma}

Note that $(B,\omega)$ is spherically positive monotone and
$c_1(\xi)=0$ whenever, for instance, $(B,\omega)$ is positive monotone
over $\Z$, i.e., $c_1(TB)=\lambda [\om]$ in $H^2(B;\Z)$ for some
integer $\lambda\geq 0$, where abusing notation we treat $[\om]$
as its lift to $H^2(B;\Z)$; cf.\ Remarks \ref{rmk:lift} and
\ref{rmk:monotone}.  Lemma \ref{lemma:pos_index} is the main point in
the proof of Theorem \ref{thm:main} where it is essential that in \ref{cond:NF}
$c_1(\xi)=0$ as an element of $H^2(M;\Z)$ and not only modulo torsion.
The lemma is not entirely new and has several predecessors (see, e.g.,
\cite[Sect.\ 2.2]{Bo:thesis} or \cite[Sect.\ 3]{GG:MW} and references
therein). However, we include a short detailed proof for the sake of
completeness and because we think the argument is a good illustration
of usefulness of the quasimorphism property of the mean index.

\begin{proof}
  Contractible closed Reeb orbits of $\alpha_0$ comprise connected
  sets $P_m$ each of which is a principal $S^1$-bundle over $B$. The
  set $P_1$ is formed by the orbits with period $a$, where as above
  $a$ is the positive generator of
  $\left<\om,\pi_2(B)\right>\subset \R$. These orbits are not
  necessarily simple but they are ``simple contractible orbits''. The
  orbits from $P_m$ are the $m$th iterations of the orbits in
  $P_1$. These orbits have mean index $2c_B m$ and period $ma$.

  Fix $T_0>0$. Then, when $\beta$ is sufficiently $C^2$-close to
  $\alpha_0$, every contractible closed Reeb orbit $\gamma$ of $\beta$
  with action $T(\gamma)\leq T_0$ is close to one of the orbits in
  $P_m$ with $ma\leq T_0$. Hence
\begin{equation}
\label{eq:index-lower-bounds}
\mu(\gamma)> 2c_Bm-n\geq 2c_B-n\textrm{ and } \hmu(\gamma)\geq \frac{2c_B}{a'}
\cdot T(\gamma)\geq \frac{c_B}{a} \cdot T(\gamma),
\end{equation}
where we can take $a'>a>0$ arbitrarily close to $a$ when $\beta$ is
close to $\alpha_0$.

Since $c_1(\xi)=0$, the determinant line bundle $\bigwedge^n_{\C}\xi$
is trivial.  Fix a section of this line bundle. Then, using this
section, we can define the mean index for all finite segments $\eta$
of Reeb orbits, not necessarily contractible or even closed, for any
contact form on $(M,\xi)$; see, e.g., \cite{Es}. This index depends
continuously on the initial condition and the contact form (in the
$C^2$-topology), and for closed contractible Reeb orbits it agrees
with the standard mean index defined in Section \ref{sec:index_orbit}.

Pick $m$ such that $2c_B m>C_n$, where $C_n$ is the quasimorphism
constant from \eqref{eq:qm}, and fix $b_0$ with $2c_Bm>b_0>C_n$. By
continuity, when $\beta$ is sufficiently $C^2$-close to $\alpha_0$,
every segment of a Reeb orbit of $\beta$ with action $ma$ has mean
index greater than $b_0$. (This is a consequence of the fact that the
Reeb orbits of $\alpha_0$ with action $m a$ have mean index $2c_B m$.)

Then, by the quasimorphism property
\eqref{eq:qm},
$$
\hmu(\eta)\geq b T(\eta)-c
$$
for all finite segments $\eta$ of Reeb orbits of $\beta$. Here we can
take $b=(b_0-C_n)/m a>0$ and $c$ depends on $m$ and the section, but
can be taken independent of $\beta$. In particular, this inequality
holds for all contractible closed Reeb orbits of $\beta$ and all such
orbits with sufficiently large action have large Conley--Zehnder
index.

Let us now take $\beta$ so close to $\alpha_0$ that
\eqref{eq:index-lower-bounds} holds for all contractible closed Reeb
orbits of $\beta$ with action smaller than a large initial time
$T_0$. To be more specific, fix a positive constant
$\Delta<\min\{b, c_B/a\}$.  Then, as is easy to see, when $T_0$ is
large enough (e.g., $T_0=c_B/(b-\Delta)$), for all contractible closed
Reeb orbits $\gamma$ of $\beta$ we have $\mu(\gamma)\geq 2c_B-n$ and
$\hmu(\gamma)\geq \Delta \cdot T(\gamma)$. This concludes the proof of
the lemma and of the proposition.
\end{proof}\end{proof}

\subsection{Local equivariant symplectic homology, resonance relation
  and Morse inequalities}
\label{sec:lesh}
Let $\ga$ be an isolated closed Reeb orbit and denote by $\HC_*(\ga)$
its local equivariant symplectic homology.  For a non-degenerate orbit
$\ga$, if $\gamma$ is good $\HC_*(\ga)=\Q$, concentrated in degree
$*=\mu(\gamma)$; $\HC_*(\ga)=0$ if $\gamma$ is bad.  The Euler
characteristic of $\ga$ is defined as
\[
\chi(\gamma) = \sum_{m \in \Z} (-1)^m\dim \HC_m(\gamma).
\]
This sum is finite. When $\gamma$ is non-degenerate,
$\chi(\gamma)=(-1)^{\mu(\gamma)}$ or $\chi(\gamma)=0$ depending on
whether $\gamma$ is good or bad.  The local {\it mean} Euler
characteristic of $\gamma$ is
\[
  \hat\chi(\gamma) = \lim_{j\to\infty} \frac{1}{j} \sum_{k=1}^j
  \chi(\gamma^k).
\]
The limit above exists and is rational; see \cite{GGo}. When $\ga$ is
strongly non-degenerate, we have
\begin{equation*}
\label{eq:lmec-nondeg}
\hat\chi(\gamma) =
\begin{cases}
(-1)^{\cz(\ga)}\text{ if }\ga^2\text{ is good} \\
(-1)^{\cz(\ga)}/2\text{ if }\ga^2\text{ is bad}.
\end{cases}
\end{equation*}

Assume now that $\alpha$ is index-positive/index-negative and has
finitely many distinct simple contractible closed orbits
$\ga_1,\dots,\ga_r$.  (Here ``simple" means that each $\ga_i$ is not
an iterate of a contractible orbit.)  This assumption ensures that the
positive/negative mean Euler characteristic
\begin{equation*}
\label{eq:def_MEC}
\chi_\pm(M) := \lim_{j\to\infty} \frac{1}{j} \sum_{m=0}^j (-1)^m b_{\pm m}
\end{equation*}
is well defined, where $b_m := \dim \HC^0_m(M)$ is the $m$-th Betti
number; see \cite{GGo}. The mean Euler characteristic is related to
local equivariant symplectic homology via the resonance relation
\begin{equation}
\label{eq:resonance}
\sum_{i=1}^{r} \frac{\hat\chi(\gamma_i)}{\mi (\gamma_i)} = \chi_\pm(M),
\end{equation}
proved in \cite{GK,HM}. Here the right-hand side is $\chi_+$ when
$\alpha$ is index-positive and $\chi_-$ when $\alpha$ is
index-negative.

Let $c_m := \sum_{i=1}^r \sum_{k=1}^\infty \dim \HC_m(\ga_i^k)$ be the
$m$-th Morse type number and define
$m_{\min} := \inf\{m \in \Z \mid c_{m} \neq 0\}$ and
$m_{\max} := \sup\{m \in \Z \mid c_{m} \neq 0\}$. When $\alpha$ is
non-degenerate, $c_m$ is simply the number of good orbits of index
$m$. Furthermore, $m_{\min}>-\infty$ if $\alpha$ is index-positive and
$m_{\max}<\infty$ if $\alpha$ is index-negative. We have the Morse
inequalities
\begin{equation}
\label{eq:Morse_ineq}
c_m - c_{m-1} + \dots \pm c_{m_{\min}} 
\geq b_m - b_{m-1} + \dots \pm b_{m_{\min}},
\end{equation}
for every $m\geq m_{\min}$ if $\alpha$ is index-positive, and
\begin{equation*}
\label{eq:Morse_ineq-negative}
c_m - c_{m+1} + \dots \pm c_{m_{\max}} 
\geq b_m - b_{m+1} + \dots \pm b_{m_{\max}},
\end{equation*}
for every $m\leq m_{\max}$ if $\alpha$ is index-negative. We note that
these inequalities are notably stronger than the inequalities
$c_m\geq b_m$.

\section{Index recurrence}
\label{sec:IRT}

\subsection{The index recurrence theorem}
A crucial ingredient for distinguishing simple and iterated orbits in
the proof of Theorem \ref{thm:main} is the following combinatorial
result addressing the index behavior under iterations. This result can
be deduced from the so-called enhanced common index jump theorem due
to Duan, Long and Wang \cite{DLW2} (see also \cite{Lon02,LZ}), but we
will give a different, self-contained proof along the lines of the
argument from \cite[Thm.\ 5.1]{GG:convex}.

\begin{theorem}
\label{thm:IRT}
Let $\Phi_1,\ldots,\Phi_r$ be a finite collection of strongly
non-degenerate elements of $\TSp(2n)$ with $\hmu(\Phi_i)>0$ for all
$i$. Then for any $\eta>0$ and any $\ell_0\in\N$, there exist two
integer sequences $d_j^\pm\to\infty$ and two sequences of integer
vectors $\vk^\pm_j=\big(k_{1j}^\pm,\ldots,k_{rj}^\pm\big)$ with all
components going to infinity as $j\to\infty$, such that for all $i$
and $j$, and all $\ell\in\Z$ in the range $1\leq |\ell|\leq \ell_0$,
we have
\begin{enumerate}%[style=multiline, labelwidth=.7cm]
\item[\reflb{cond:i}{\rm{(i)}}]
  $\big|\hmu\big(\Phi^{k_{ij}^\pm}_i\big)-d_j^\pm\big|<\eta$ with the
  equality
  $\hmu\big(\Phi^{k_{ij}^\pm}_i\big) =
  \cz\big(\Phi^{k_{ij}^\pm}_i\big)=d_j^\pm$ whenever $\Phi_i(1)$ is
  hyperbolic,
\item[\reflb{cond:ii}{\rm{(ii)}}]
  $\cz\big(\Phi^{k_{ij}^\pm+\ell}_i\big)= d_j^\pm + \cz(\Phi^\ell_i)$,
  and
\item[\reflb{cond:iii}{\rm{(iii)}}]
  $\cz\big(\Phi^{k_{ij}^-}_i\big)-d_j^-=
  -\big(\cz\big(\Phi^{k_{ij}^+}_i\big)-d_j^+\big)$.
\end{enumerate}
Furthermore, for any $N\in \N$ we can make all $d_j^\pm$
and $k_{ij}^\pm$ divisible by~$N$.
\end{theorem}

The condition that $\hmu(\Phi_i)>0$ for all $i$ can be relaxed, but
the theorem, as is, is sufficient for our purposes.

In the assertion and the proof of the theorem we follow closely
\cite[Sect.\ 5]{GG:convex}. (The new point is (iii); the rest is
contained in, e.g., \cite[Thm.\ 5.1]{GG:convex}.) Note that it
suffices to find just one pair $\vk^\pm=(k_1^\pm,\ldots,k_r^\pm)$ of
iteration vectors and one pair $d^\pm$, both divisible by any given
$N$ --- and this is the form of the theorem we actually use here. Once
$\vk^\pm_1=\vk^\pm$ and $d_1^\pm=d^\pm$ are found we can replace $N$
by $pN$, where $p$ is a sufficiently large integer, and repeat the
process to find $\vk^\pm_2$ and $d^\pm_2$, and so on.

\subsection{Proof of Theorem \ref{thm:IRT}}
\label{sec:IRT-pf}
We first establish the case of a single path $\Phi$, i.e., $r=1$, and
then show how to modify the argument for a finite collection of paths.

\subsubsection{The case of $r=1$.}
\label{sec:IRT-pf-r=1} 
Let $\Phi=\Phi_1\in\TSp(2n)$. Throughout the argument we suppress $i$
in the notation, i.e., we write $k$ for $k_{11}$ or $\vk_1$, etc. To prove
the theorem in this setting, we will consider two subcases depending
on the end-map $\Phi(1)$ and then derive the general case from
additivity.  Fix $\eta>0$ and $\ell_0\in\N$. Without loss of
generality, we can assume that $\eta<1/2$.

\smallskip\noindent\emph{Subcase A: $\Phi(1)$ is hyperbolic.} Set
$d_k^\pm=\hmu(\Phi^k)$ for any $k\in\N$. Clearly, (i) is automatically
satisfied. Furthermore, $\Phi^k$ is non-degenerate for all $k\in \N$
and $\hmu(\Phi^k)=\cz(\Phi^k)$. Hence we have
$$
\cz(\Phi^{k+\ell})=\hmu(\Phi^{k})+\hmu(\Phi^{\ell})
=\cz(\Phi^{k})+\cz(\Phi^{\ell}).
$$
Thus (i)--(iii) hold for all $k^\pm=k\in\N$ and all $\ell$, with
$d_k^\pm=d_k$. To make $d$ and $k$ divisible by $N$ it suffices to
just take $k$ divisible by $N$.

\smallskip\noindent\emph{Subcase B: $\Phi(1)$ is elliptic.} Let
$\exp\big(\pm 2\pi\sqrt{-1}\lambda_q\big)$, for $q=1,\ldots,n$, be the
eigenvalues of $\Phi(1)\in\Sp(2n)$, where $|\lambda_q|<1$. (The choice
of sign for $\lambda_q$ is not essential, but when the eigenvalues are
distinct it is convenient to assume that
$\exp\big(2\pi\sqrt{-1}\lambda_q\big)$ are the eigenvalues of the
first kind; see, e.g., \cite{SZ}.)  Since $\Phi(1)$ is strongly
non-degenerate, all $\lambda_q$ are irrational. Set
\begin{equation}
\label{eq:eps1}
\ep_0=\min_{0<\ell\leq\ell_0}\min_q\|\lambda_q\ell\|>0,
\end{equation}
where $\|\cdot\|$ stands for the distance to the nearest integer. Let
$\ep>0$ be so small that
$$
\ep\leq\ep_0\quad\textrm{and}\quad n\ep<\eta.
$$
It is easy to see that there exists $k>0$ such
that for all $q$ we have
\begin{equation}
\label{eq:eps2}
\|\lambda_q k\|<\ep\leq \ep_0.
\end{equation}
Indeed, consider the positive semi-orbit
$\Gamma_+=\{k\vec{\lambda}\mid k\in \N\}\subset \T^{n}$ where
$\vec{\lambda}\in\T^{n}$ is the collection of eigenvalues of
$\Phi(1)$. As is well known, the closure $\Gamma$ of $\Gamma_+$ is a
subgroup of $\T^{n}$. Hence $\Gamma_+$ contains points arbitrarily
close to the unit in $\T^{n}$ and, in particular, there exist
infinitely many points $k\vec{\lambda}\in\Gamma_+$ in the
$2\pi\ep$-neighborhood of the unit. Clearly, for any $N\in\N$ we can
also make $k$ divisible by $N$. (To see this, it suffices to replace
the semi-orbit $\Gamma_+$ by $\{kN\vec{\lambda}\mid k\in \N\}$.)

Let $d$ be the nearest integer to $\hmu(\Phi^{k})$. Then
\begin{equation}
\label{eq:d}
\big|d-\hmu(\Phi^{k})\big|\leq n\ep<\eta,
\end{equation}
and hence \ref{cond:i} %(i)
is satisfied. (This also shows that $d$ is unambiguously defined.)
Furthermore, replacing as above the semi-orbit $\Gamma_+$ by
$\{kN\vec{\lambda}\mid k\in \N\}$ we can also make $d$ divisible by
$N$.

It is shown in \cite[Sect.\ 5.2.1, Subcase C]{GG:convex} that the
inequalities \eqref{eq:eps2} and \eqref{eq:d} imply (ii). For the sake
of completeness we recall here the argument. Observe first that a
small perturbation of $\Phi$ does not effect individual terms in these
inequalities for fixed $k$ and $\ell$. Thus, by altering $\Phi$
slightly, we can ensure that all eigenvalues $\lambda_q$ are
distinct. Then we can write $\Phi$, up to homotopy, as the product of
a loop $\varphi$ and the direct sum of paths
$\Psi_q=\exp(2\pi \sqrt{-1} \lambda_q t)\in\TSp(2)$ for a suitable
choice of signs of $\lambda_q$; see, e.g., \cite[Sect.\ 3]{SZ}. The
loop $\varphi$ contributes $k\hmu(\varphi)$ to $\cz(\Phi^k)$ and hence
we only need to prove (ii) when $\varphi=I$.

Then, for any $k$,
$$
\cz(\Phi^k)=\sum_q\cz(\Psi_q^k).
$$
Next, observe that by \eqref{eq:eps1} and \eqref{eq:eps2} we have
$$
d=\sum_q [\hmu(\Psi_q)],
$$
where $[\,\cdot\,]$ denotes the nearest integer. Thus it suffices to
prove (ii) and (iii) for each path $\Psi_q$ individually when we set
$d_q=[\hmu(\Psi_q)]$.  However, with \eqref{eq:eps1} and
\eqref{eq:eps2} in mind, (ii) for $\Psi_q$ easily follows from the
definition.

Now we need to find $k^-$ satisfying (ii) for a suitable choice of
$d^-$. To this end, observe that for any $\delta>0$ the system of
inequalities
\begin{equation}
\label{eq:delta}
\left\|\lambda_q (k^-+k^+)\right\|<\delta
\end{equation}
has infinitely many solutions $k^-\in\N$, where $k^+:=k$. This is
again a consequence of the fact that $\Gamma_+$ is dense in the group
$\Gamma$, and hence contains points arbitrarily close to
$-k^+\vec{\lambda}$ in $\T^n$. It is also clear that $k^-$ can be
made divisible by any given integer. Using \eqref{eq:eps2}, let us
take $\delta$ so small that $\|k^+\lambda_q\|+\delta<\ep$.

Then \eqref{eq:eps2} is still satisfied for $k=k^-$. Let $d^-$ be the
nearest integer to $\hmu(\Phi^{k^-})$. Then \eqref{eq:eps2} and
\eqref{eq:d} hold for $k^-$ and $d^-$ and hence so does (ii). Finally,
since $k^-\vec{\lambda}$ is close to $-k^+\vec{\lambda}$ in the torus
$\T^n$, we have
$$
\cz\big(\Phi^{k^-}\big)-d^-=\cz\big(\Phi^{-k^+}\big)+d^+
=-\big(\cz\big(\Phi^{k^+}\big)-d^+\big),
$$
which proves (iii).

\smallskip\noindent\emph{Putting Subcases A--B together.} Let us
decompose $\Phi$ into the direct sum of two paths $\Phi_A$ and
$\Phi_B$ such that $\Phi_A(1)$ is hyperbolic and $\Phi_B(1)$ is
elliptic. (It is easy to see that we can always do this up to
homotopy.) We take $k^\pm$ as in Subcase B and adjust $d^\pm$ by
adding $\hmu\big(\Phi_A^{k^\pm}\big)$. It is clear that (i)--(iii)
hold for this choice of $k^\pm$ and $d^\pm$ and that, in addition, we
can make $k^\pm$ and $d^\pm$ divisible by any integer.

\subsubsection{The general case: $r\geq 1$.}
\label{sec:CGT-pf-gen} 
Let $\Phi_1,\dotsc,\Phi_r$ be a finite collection of elements in
$\TSp(2n)$. As above, each of these paths can be decomposed into a sum
of paths with hyperbolic end-points and elliptic end-points. Then it
is easy to see that it suffices to prove the theorem when all
$\Phi_i(1)$ are elliptic. For the general case follows again by
additivity.

Denote by $\exp\big(\pm 2\pi\sqrt{-1}\lambda_{iq}\big)$ the
eigenvalues of $\Phi_i$ with $|\lambda_{iq}|<1$ and set
$\Delta_i=\hmu(\Phi_i)>0$. (The choice of the sign of $\lambda_{iq}$
is immaterial at the moment, but again when all eigenvalues are
distinct it is convenient to assume that
$\exp\big(2\pi\sqrt{-1}\lambda_{iq}\big)$ are the eigenvalues of the
first kind.)  Given $\ep>0$, consider the system of inequalities
\begin{align}
  \|k_i\lambda_{iq}\|& <\ep \quad 
   \textrm{ for all $i$ and $q$,}\label{eq:k1}\\
  |k_1\Delta_1-k_{i}\Delta_{i}| &< \frac{1}{16} \quad
\textrm{ for $i=2,\dotsc,r$,}\label{eq:k2}
\end{align}
where we treat the integer vector $\vec{k}=(k_1,\dotsc, k_r)\in\Z^r$
as a variable. Introducing additional integer variables $c_{iq}$, we
can rewrite \eqref{eq:k1} in the form
\begin{equation}
\label{eq:ciq}
|k_i\lambda_{iq}-c_{iq}|<\ep.
\end{equation}
With this in mind, the system of equations \eqref{eq:k1} and
\eqref{eq:k2}, or equivalently \eqref{eq:k2} and \eqref{eq:ciq}, has
one fewer equation than the number of variables. By Minkowski's
theorem (see, e.g., \cite{Ca}), there exists a non-zero solution
$\vec{k}= (k_{1},\dotsc, k_{r})$ of \eqref{eq:k2} and
\eqref{eq:ciq}. Now it follows from \eqref{eq:k2} and the assumption
that $\Delta_i>0$ that all $k_i$ are non-zero and have the same
sign. Hence, replacing if necessary $\vk$ by $-\vk$, we can ensure
that $k_i>0$. Moreover, we can make all $k_{i}$ divisible by any fixed
integer $N$.

Note also that by \eqref{eq:k2} we have
$$
\Big|\sum\nolimits_q c_{1q} -\sum\nolimits_q
c_{iq}\Big|<\frac{1}{16}+2r\ep.
$$ 
If $\ep<1/4r$, this inequality is satisfied only when the left-hand
side is zero.

Fix $\ell_0$ and $\eta>0$ which we assume to be sufficiently small
(e.g., $\eta<1/4$).  Similarly to Subcase B, set
$$
\ep_0=\min_{0<\ell\leq\ell_0}\min_{i,q}\|\lambda_{iq}\ell\|>0,
$$
and let $\ep>0$ be so small that again
$$
\ep\leq\ep_0\quad\textrm{and}\quad n\ep<\eta.
$$

By \eqref{eq:k2} we have
$$
|k_i\Delta_i-k_{i'}\Delta_{i'}|<\frac{1}{8} \quad \textrm{for all $i$
  and $i'$,}
$$
and $\|k_{i}\Delta_i\|<\eta$ by the first group of inequalities
\eqref{eq:k1}. Thus $k_{i}\Delta_i$ is $\eta$-close, for all $i$, to
the same integer
$$
d=[k_{i}\Delta_i]=\sum\nolimits_q c_{iq}=\sum\nolimits_q c_{1q}.
$$
In other words, (i) is satisfied for this choice of $d$.  Furthermore,
for every $i$, condition \eqref{eq:eps2} is met for $\lambda_{iq}$, and
hence (ii) holds for all $\Phi_i$. As in Subcase B, we set $d^+=d$ and
$\vk^+=\vk$. Note that so far it would be sufficient to take $1/8$ as
the right-hand side in \eqref{eq:k2}.

Our next goal is to find $d^-$ and $\vk^-$. To this end consider the
inequalities
\begin{align}
\|k'_i\lambda_{iq}\|& <\delta 
\quad\textrm{ for all $i$ and $q$,}\label{eq:k'1}\\
  |k'_1\Delta_1-k'_{i}\Delta_{i}| &< \frac{1}{16} 
\quad\textrm{ for $i=2,\dotsc,r$,}\label{eq:k'2}
\end{align}
where \eqref{eq:k'1} can again be written in the form
\begin{equation}
\label{eq:ciq'}
\big|k'_i\lambda_{iq}-c'_{iq}\big|<\delta
\end{equation}
% with some unknown integer vector $\{c'_{iq}\}$.
for some integer variables $c'_{iq}$. For any $\delta>0$, the system
of inequalities \eqref{eq:k'2} and \eqref{eq:ciq'} has a non-zero
solution $\vk'$ by Minkowski's theorem. The same argument as above
shows that we can take $k'_i>0$ for all $i$ and, in fact, we can make
$k'_i$ arbitrarily large. In particular, we can ensure that
$$
k_i^-:=k'_i-k_i^+>0.
$$
Then we have
$$
\left\|(k_i^-+k_i^+)\lambda_{iq}\right\| <\delta \quad\textrm{ for all
  $i$ and $q$}
$$
and
$$
|k_1^-\Delta_1-k_{i}^-\Delta_{i}| < \frac{1}{8} \quad\textrm{ for
    $i=2,\dotsc,r$,}
$$
where to obtain the last inequality we used \eqref{eq:k2} and
\eqref{eq:k'2}. Let us now assume that $\delta>0$ is so small that
$\|k_{iq}^+\lambda_{iq}\|+\delta<\ep$ for all $i$ and $q$. Then we
also have
$$
\|k_{iq}^-\lambda_{iq}\|<\ep \quad\textrm{ for all $i$ and $q$,}
$$
i.e., \eqref{
eq:k1} holds for $\vk^-$. 

To summarize, $\vk^-$ satisfies \eqref{eq:k1} and \eqref{eq:k2} with
$1/16$ replaced by $1/8$, which is sufficient for our
purposes. Setting
$$
d^-=[k_{i}^-\Delta_i]=\sum\nolimits_q (c_{iq}'-c_{iq})
$$
we conclude that (i) and (ii) hold for $\vk^-$. It is also clear that
we can make, if necessary, all $k_i^-$ divisible by an arbitrary
constant $N$.

Finally, (iii) also holds for each $\Phi_i$ individually just as in
Subcase B since the vector $\vk^-$ is close to $-\vk^+$ modulo the
integer lattice. This concludes the proof of Theorem
\ref{thm:IRT}. \qed

\section{Proofs of Theorems \ref{thm:main} and \ref{thm:app}}
\subsection{Proof of Theorem \ref{thm:main}}
\label{sec:pf-main}
Let us focus on the case where $\alpha$ is index-positive, for the
argument in the index-negative case is similar. The main tool used in
the proof is the positive equivariant symplectic homology. Recall from
Section \ref{sec:esh} that both of the conditions \ref{cond:F} and
\ref{cond:NF} ensure that this homology (for $M$ or the filling) with
integer grading is defined and, as Proposition \ref{prop:CH} shows,
given by \eqref{eq:CH}. Then the proof is the same in both cases of
the theorem, \ref{cond:F} and \ref{cond:NF}, and relies only on the
condition shared by these cases that $\alpha$ has no good contractible
periodic orbits $\ga$ such that $\cz(\ga)=0$ if $n$ is odd or
$\cz(\ga) \in \{0,\pm 1\}$ if $n$ is even. We should note that the
argument also uses several ideas from \cite{DLW2, DLLW}.

Starting the proof, assume that $\alpha$ has finitely many distinct
contractible simple closed orbits $\ga_1,\dots,\ga_r$. (Here, as in Section \ref{sec:lesh}, ``simple"
means that each $\ga_i$ is not an iterate of a contractible orbit.)
Our goal is to establish the lower bound on $r$ asserted by the
theorem. Define
\[
  \lo = \max_{1 \leq i \leq r} \min\{k_0 \in \N \mid \cz(\ga_i^{k+\l})
  \geq \cz(\ga_i^k)+2n+1 \text{ for all } k \geq 1 \text{ and } \l
  \geq k_0\}.
\]
By Theorem \ref{thm:IRT}, given $N \in \N$, $\eta>0$ and $\lo$ as
above we have two sequences of integer vectors
$(d^\pm_j,k^\pm_{1j},\dots,k^\pm_{rj})$ satisfying conditions
\ref{cond:i}, \ref{cond:ii} and \ref{cond:iii} and such that all
$d^\pm_j,k^\pm_{1j},\dots,k^\pm_{rj}$ are divisible by $N$. As has
been mentioned before, we will only need one such vector from each
sequence. Hence set
\[
  (d,k_1,\dots,k_r):=(d^+_1,k^+_{11},\dots,k^+_{r1})\text{ and
  }(d',k'_1,\dots,k'_r):=(d^-_1,k^-_{11},\dots,k^-_{r1}).
\]
The following lemma, giving an expression for the truncated mean Euler
characteristic (see Section \ref{sec:lesh}), is one of the key steps
in the proof; cf.\ \cite[Sublemma 5.2]{AM:multiplicity}.

\begin{lemma}
\label{lemma:resonance}
The numbers $N$ and $\eta$ can be chosen such that $d=2sc_B$ for some
integer $s$ and
\[
  \sum_{i=1}^r \sum_{k=1}^{k_{i}} \chi(\ga_i^k) = \sum_{i=1}^r k_{i}
  \hat{\chi} (\gamma_i) = d \chi_+(M) = (-1)^n s\chi(B).
\]
The same holds for $d',k'_1,\dots,k'_r$.
\end{lemma}

\begin{proof}
  Let $N$ be any (positive) integer multiple of $2c_B$. The first
  equality follows from the (strong) non-degeneracy of
  $\ga_1,\dots,\ga_r$ since the numbers $k_i$ are even. It is easy to
  see from \eqref{eq:CH} that
\[
\chi_+(M) = (-1)^n \frac{\chi (B)}{2c_B},
\]
which implies the third equality. To prove the second equality, take
$\eta$ sufficiently small such that $\eta|\chi_+(M)|<1$. Using the
resonance relation \eqref{eq:resonance}, we conclude that
\begin{eqnarray*}
d \chi_+(M) 
& = & \sum_{i=1}^r \frac{d\hat\chi (\gamma_i)}{\mi (\gamma_i)} \\
& = & \sum_{i=1}^r  k_i\hat\chi (\gamma_i) +  
      \sum_{i=1}^r  \frac{(d-k_i\mi(\ga_i))
      \hat\chi(\gamma_i)}{\mi(\ga_i)} \\
& = & \sum_{i=1}^r  k_i\hat\chi (\gamma_i) +  
      \sum_{i=1}^r  \frac{(d-\mi(\ga_i^{k_i}))
      \hat\chi (\gamma_i)}{\mi(\ga_i)}. 
\end{eqnarray*}
By property \ref{cond:i} of Theorem \ref{thm:IRT} and, again,
\eqref{eq:resonance},
\[
  \bigg| \sum_{i=1}^r \frac{(d-\mi(\ga_i^{k_i}))\hat\chi
    (\gamma_i)}{\mi(\ga_i)}\bigg| < \eta \bigg| \sum_{i=1}^r
  \frac{\hat\chi (\gamma_i)}{\mi(\ga_i)}\bigg| = \eta|\chi_+(M)|<1.
\]
Note that by our choice of $N$ the numbers $d\chi_+(M)$ and
$k_i\hat\chi (\gamma_i)$ for all $i$ are integers.  Therefore,
\[
d \chi_+(M) =  \sum_{i=1}^r  k_i\hat\chi (\gamma_i).
\]
Obviously, the same argument works for $d',k'_1,\dots,k'_r$.
\end{proof}

Let us now break down the proof of Theorem \ref{thm:main} into two
cases, according to the parity of~$n$.

\vskip .3cm

\noindent {\bf Case 1: $n$ is odd.}

\vskip .2cm

Fix $N$ and $\eta$ as in Lemma \ref{lemma:resonance}. (In particular,
$N$ is even.)  Clearly, $\eta$ can be chosen so small that the vector
$(d,k_{1},\dots,k_{r})=(d^+_1,k^+_{11},\dots,k^+_{r1})$ given by
Theorem \ref{thm:IRT} satisfies
\begin{equation}
\label{eq:IRT1}
\cz(\ga_i^{k_i-\l}) = d - \cz(\ga_i^\l),
\end{equation}
\begin{equation}
\label{eq:IRT2}
\cz(\ga_i^{k_i+\l}) = d + \cz(\ga_i^\l),
\end{equation}
and
\begin{equation}
\label{eq:IRT3}
|\cz(\ga_i^{k_i}) - d| \leq n
\end{equation}
for every $1\leq i \leq r$ and $1 \leq \l \leq \lo$.  (Here, since $N$
is even, the integers $d$ and $k_i$ are even. One can also assume that
$k_i > \lo$ for all $i$.) Observe that for each periodic orbit
$\ga_i$, there are four types of iterates outside $\ga_i^{k_i}$:
\begin{itemize}
\item[(A)] $\ga_i^{k_i-\l}$ with $\l>\lo$;
\item[(B)] $\ga_i^{k_i-\l}$ with $1\leq \l \leq \lo$;
\item[(C)] $\ga_i^{k_i+\l}$ with $1\leq \l \leq \lo$;
\item[(D)] $\ga_i^{k_i+\l}$ with $\l>\lo$.
\end{itemize}
Let us analyze the contributions of these iterates to the Morse type
numbers defined by the alternating sum
\begin{equation}
\label{eq:c_m}
\sum_{m=m_{\min}}^{d} (-1)^m c_m,
\end{equation}
where, as in Section \ref{sec:lesh}, $c_m$ is the number of good
closed orbits of index $m$ and $m_{\min}>-\infty$ is the smallest
integer with $c_{m} \neq 0$.  First, class (A) iterates have index
$<d$ and hence all good orbits here contribute to \eqref{eq:c_m}.
Indeed, by definition of $\lo$, we have
$\cz(\ga_i^{k_i}) \geq \cz(\ga_i^{k_i-\l}) + 2n+1$ for every $\l>\lo$
which, combined with \eqref{eq:IRT3}, implies that
$\cz(\ga_i^{k_i-\l}) \leq d - n-1$ for all $\l>\lo$.  Class (D) orbits
do not contribute to \eqref{eq:c_m} since
$\cz(\ga_i^{k_i+\l}) \geq \cz(\ga_i^{k_i}) + 2n+1 \geq d + n+1$ for
every $\l>\lo$, where the last inequality again follows from
\eqref{eq:IRT3}.

In order to understand the contributions from classes (B) and (C), let
us further divide each of them into two subclasses:

\begin{itemize}
\item[(B1)] $\ga_i^{k_i-\l}$ with $1\leq \l \leq \lo$ if
  $\cz(\ga_i^\l) \geq 0$,
\item[(B2)] $\ga_i^{k_i-\l}$ with $1\leq \l \leq \lo$ if
  $\cz(\ga_i^\l) < 0$,
\end{itemize}
and
\begin{itemize}
\item[(C1)] $\ga_i^{k_i+\l}$ with $1\leq \l \leq \lo$ if
  $\cz(\ga_i^\l) \geq 0$,
\item[(C2)] $\ga_i^{k_i+\l}$ with $1\leq \l \leq \lo$ if
  $\cz(\ga_i^\l) < 0$.
\end{itemize}
Now all of the good orbits in class (B1) contribute to \eqref{eq:c_m},
while class (B2) makes no contribution to \eqref{eq:c_m}. Indeed, by
\eqref{eq:IRT1}, $\cz(\ga_i^{k_i-\l}) = d - \cz(\ga_i^\l)$ which is
$\leq d$ whenever $\cz(\ga_i^\l) \geq 0$ and $>d$ whenever
$\cz(\ga_i^\l) < 0$.

The key to dealing with class (C1) is the condition that $\alpha$ has
no good contractible periodic orbits of index zero.  (This is the main
point where this condition is used.) Then for all good iterates
$ \ga_i^{k_i+\l}$ in class (C1) $\cz(\ga_i^\l) > 0$ and, by
\eqref{eq:IRT2}, $\cz(\ga_i^{k_i+\l}) = d + \cz(\ga_i^\l) > d$
whenever $\ga_i^{k_i+\l}$ is good.  Finally, all of the good orbits
from class (C2) contribute to \eqref{eq:c_m} since $\cz(\ga_i^\l) < 0$
and $\cz(\ga_i^{k_i+\l}) = d + \cz(\ga_i^\l) < d$ by \eqref{eq:IRT2}.
(Above $\cz(\ga_i^{k_i+\l})$ and $\cz(\ga_i^\l)$ have the same parity
since the integers $k_i$ are even.)

To summarize, all good orbits from classes (A), (B1) and (C2) have
index $\leq d$ and contribute to \eqref{eq:c_m}, and good orbits from
classes (D), (B2) and (C1) have index $> d$ and make no contribution
to \eqref{eq:c_m}.  Define
\begin{equation}
\label{eq:c^e}
c^e_\pm = \sum_{i=1}^r \#\{1 \leq \l \leq \lo \mid \cz(\ga_i^\l) < 0,\ 
\ga_i^{k_i\pm \l}\text{ is good and }\cz(\ga_i^\l)\text{ is even}\}
\end{equation}
and
\begin{equation}
\label{eq:c^o}
c^o_\pm = \sum_{i=1}^r \#\{1 \leq \l \leq \lo \mid \cz(\ga_i^\l) < 0,\ 
\ga_i^{k_i\pm \l}\text{ is good and }\cz(\ga_i^\l)\text{ is odd}\}.
\end{equation}
Consider now $ \sum_{m=m_{\min}}^\infty (-1)^m c_m $ to which all good
orbits from classes (A)--(D) and the collection $\{\gamma_i^{k_i}\}$
contribute. In particular, the contributions of class (B2) and class
(C2) iterates are respectively $c_-^e-c_-^o$ and
$c_+^e-c_+^o$. Viewing \eqref{eq:c_m} as
$$ 
\sum_{m=m_{\min}}^{d} (-1)^m c_m=\sum_{m=m_{\min}}^\infty (-1)^m c_m -
\sum_{m>d} (-1)^m c_m,
$$ 
with the above discussion in mind, we have
\begin{equation}
\label{eq:contributions1}
\sum_{m=m_{\min}}^{d} (-1)^m c_m 
= \sum_{i=1}^r\bigg(\underbrace{\sum_{\l=1}^{k_i}
  \chi(\ga_i^\l)}_{(A)+(B)+\chi(\ga_i^{k_i})} 
+ \underbrace{c_+^e-c_+^o}_{(C2)}
- \underbrace{(c_-^e-c_-^o)}_{(B2)} 
- \sum_{\cz(\ga_i^{k_i})>d} \chi(\ga_i^{k_i})\bigg).
\end{equation}
Here, as indicated by the underbraces, the first term on the
right-hand side comes from the iterates in classes (A) and (B) and the
iterate $\ga_i^{k_i}$, the second term comes from class (C2) iterates,
and the third term cancels out the contribution of class (B2) orbits
to the first term. Finally, the last term eliminates the contribution
to the first term of the orbits $\ga_i^{k_i}$ with index greater than
$d$.

Note that, since $\cz(\ga^{k_i-\l})$ and $\cz(\ga^{k_i+\l})$ have the
same parity, $c_-^e = c_+^e$ and $c_-^o = c_+^o$.  Thus the second and
third terms on the right-hand side of equation
\eqref{eq:contributions1} cancel each other out and we arrive at
\begin{equation}
\label{eq:contributions2}
\sum_{m=m_{\min}}^{d} (-1)^m c_m 
= \sum_{i=1}^r\sum_{\l=1}^{k_i} \chi(\ga_i^\l) 
- \sum_{i=1}^r\sum_{\cz(\ga_i^{k_i})>d} \chi(\ga_i^{k_i}).
\end{equation}
Define
\begin{equation}
\label{eq:r^e}
r^e_\pm = \#\{1 \leq i \leq r \mid 
\pm(\cz(\ga_i^{k_i}) - d) > 0,\ 
\ga_i^{k_i}\text{ is good and }\cz(\ga_i^{k_i})\text{ is even}\}
\end{equation}
and
\begin{equation}
\label{eq:r^o}
r^o_\pm = \#\{1 \leq i \leq r \mid 
\pm(\cz(\ga_i^{k_i}) - d) > 0,\ 
\ga_i^{k_i}\text{ is good and }\cz(\ga_i^{k_i})\text{ is odd}\}.
\end{equation}
Notice that the last term in \eqref{eq:contributions2} is given by
$r^e_+ - r^o_+$.  Then, if we write $d=2sc_B$, equation
\eqref{eq:contributions2}, together with Lemma \ref{lemma:resonance},
yields the relation
\begin{align*}
\label{eq:morse_ineq}
  -s\chi(B) - r^e_+ + r^o_+ 
  & = \sum_{m=m_{\min}}^{d} (-1)^m c_m \\
  & \geq \sum_{m=m_{\min}}^{d} (-1)^m b_m \\
  & = -s\chi(B) + \sum_{i=0}^{n-1} (-1)^i \dim H_i(B;\Q),
\end{align*}
where we have used the assumption that $n$ is odd. The inequality
follows from the Morse inequalities \eqref{eq:Morse_ineq} and the last
equality follows from \eqref{eq:CH} using the hypothesis that
$c_B>n/2$. Hence
\begin{equation}
\label{eq:lower estimate r^o_+}
r^o_+ \geq \sum_{i=0}^{n-1} (-1)^i \dim H_i(B;\Q).
\end{equation}
Now, we claim that
\begin{equation}
\label{eq:lower estimate r^o_-}
r^o_- \geq \sum_{i=0}^{n-1} (-1)^i \dim H_i(B;\Q).
\end{equation}
In order to prove this, observe that applying Theorem \ref{thm:IRT} to
$N$, $\eta$ and $\lo$ as above, we obtain positive integers
$(d',k'_1,\dots,k'_r)=(d^-_1,k^-_{11},\dots,k^-_{r1})$ such that
\begin{equation*}
\label{eq:IRT1'}
\cz(\ga_i^{k'_i-\l}) = d' - \cz(\ga_i^\l),
\end{equation*}
\begin{equation*}
\label{eq:IRT2'}
\cz(\ga_i^{k'_i+\l}) = d' + \cz(\ga_i^\l),
\end{equation*}
and
\begin{equation}
\label{eq:IRT3'}
\cz(\ga_i^{k'_i}) - d' = -(\cz(\ga_i^{k_i}) - d),
\end{equation}
for every $1\leq i \leq r$ and $1 \leq \l \leq \lo$. Arguing as
before, we arrive at the equation
\begin{equation}
\label{eq:contributions3}
\sum_{m=m_{\min}}^{d'} (-1)^m c_m 
= \sum_{i=1}^r\sum_{\l=1}^{k'_i} \chi(\ga_i^\l) - r'^e_+ + r'^o_+\,,
\end{equation}
where, similarly to \eqref{eq:r^e} and \eqref{eq:r^o},
\begin{equation*}
\label{eq:r'^e}
r'^e_\pm = \#\{1 \leq i \leq r \mid 
\pm(\cz(\ga_i^{k'_i}) - d') > 0,\ 
\ga_i^{k'_i}\text{ is good and }\cz(\ga_i^{k'_i})\text{ is even}\}
\end{equation*}
and
\begin{equation*}
\label{eq:r'^o}
r'^o_\pm = \#\{1 \leq i \leq r \mid 
\pm(\cz(\ga_i^{k'_i}) - d') > 0,\ 
\ga_i^{k'_i}\text{ is good and }\cz(\ga_i^{k'_i})\text{ is odd}\}.
\end{equation*}
Notice that, due to \eqref{eq:IRT3'}, $r^e_\pm = r'^e_\mp$ and
$r^o_\pm = r'^o_\mp$.  Therefore, if we write $d'=2s'c_B$ for some
integer $s'$, equation \eqref{eq:contributions3}, together with Lemma
\ref{lemma:resonance}, gives rise to the relation
\begin{align*}
\label{eq:morse_ineq2}
-s'\chi(B) - r^e_- + r^o_- 
& = \sum_{m=m_{\min}}^{d'} (-1)^m c_m \\
& \geq \sum_{m=m_{\min}}^{d'} (-1)^m b_m \\
& = -s'\chi(B) + \sum_{i=0}^{n-1} (-1)^i \dim H_i(B;\Q),
\end{align*}
where the assumptions that $n$ is odd and $c_B>n/2$ have once more
entered the picture. In particular, \eqref{eq:lower estimate r^o_-}
holds.

Since $r^o_+$ and $r^o_-$ count two disjoint sets of orbits, in view
of \eqref{eq:lower estimate r^o_+} and \eqref{eq:lower estimate
  r^o_-}, we must have at least $j$ distinct contractible simple
closed orbits, say, $\ga_1,\dots,\ga_{j}$, where
\[
  j:= 2\sum_{i=0}^{n-1} (-1)^i \dim H_i(B;\Q) = \chi(B) + \dim H_n(B;
  \Q)
\]
as an immediate consequence of Poincar\'e duality and the assumption
that $n$ is odd.  We claim that (good) iterates of these orbits have
index different from $d$ and hence do not contribute to
$\HC^0_{d}(M)$. Indeed, since $\cz(\ga_i^\l) \neq 0$ for every
$1 \leq i \leq r$ and $\l \in \N$ such that $\ga_i^\l$ is good, we
infer from \eqref{eq:IRT1}, \eqref{eq:IRT2} and the definition of
$\lo$ that
\[
\cz(\ga_i^\l) \neq d
\]
for every $\l \neq k_i$ and $1 \leq i \leq r$ such that $\ga_i^\l$ is
good. Therefore, only the orbits $\ga_1^{k_1},\dots,\ga_r^{k_r}$ can
contribute to $\HC^0_{d}(M)$.  However, the definition of $r^o_\pm$
given by \eqref{eq:r^o} implies that $\cz(\ga_i^{k_i}) \neq d$ for all
$1 \leq i \leq j$.  Hence
\[
r \geq j + \dim \HC^0_{d}(M)
  \geq j + \dim H_n(B;\Q),
\]
where the second inequality follows from \eqref{eq:CH}.  Finally, we
conclude that
\[
  r \geq 2\sum_{i=0}^{n-1} (-1)^i \dim H_i(B; \Q) + \dim H_n(B; \Q) =
  \chi(B) + 2\dim H_n(B; \Q),
\]
which proves Theorem \ref{thm:main} when $n$ is odd.

\begin{remark}
\label{rmk:non-hyp1}
It is clear from the definition of $r^o_\pm$ and item \ref{cond:i} of
Theorem \ref{thm:IRT} that the orbits $\ga_1,\dots,\ga_{j}$ are
non-hyperbolic. In other words, if $\alpha$ has finitely many distinct
contractible simple closed orbits then at least
$\chi(B) + \dim H_n(B)$ of them are non-hyperbolic. This establishes
Theorem \ref{thm:non-hyp} when $n$ is odd.
\end{remark}

\vskip .3cm

\noindent {\bf Case 2: $n$ is even.}

\vskip .2cm

As in the previous case, fix $N$ and $\eta$ as in Lemma
\ref{lemma:resonance}, and an integer vector
$(d,k_{1},\dots,k_{r})=(d^+_1,k^+_{11},\dots,k^+_{r1})$ as in Theorem
\ref{thm:IRT}. The argument is very similar to the one for odd
$n$. Namely, we consider, for each periodic orbit $\ga_i$, the same
classes of iterates (A), (B), (C) and (D), and study their
contributions to the Morse type numbers defined by
\begin{equation}
\label{eq:c_me}
\sum_{m=m_{\min}}^{d+1} (-1)^m c_m.
\end{equation}

Due to the same index reasons as in Case 1, all good orbits in class
(A) contribute to \eqref{eq:c_me} and class (D) orbits do not
contribute to \eqref{eq:c_me}.  To deal with classes (B) and (C), we
again consider four subclasses, although this time the index
breakpoint is $-1$, rather than $0$:

\begin{itemize}
\item[(B1)] $\ga_i^{k_i-\l}$ with $1\leq \l \leq \lo$ 
if $\cz(\ga_i^\l) \geq -1$,
\item[(B2)] $\ga_i^{k_i-\l}$ with $1\leq \l \leq \lo$ 
if $\cz(\ga_i^\l) < -1$,
\end{itemize}
and
\begin{itemize}
\item[(C1)] $\ga_i^{k_i+\l}$ with $1\leq \l \leq \lo$ 
if $\cz(\ga_i^\l) \geq -1$,
\item[(C2)] $\ga_i^{k_i+\l}$ with $1\leq \l \leq \lo$ 
if $\cz(\ga_i^\l) < -1$.
\end{itemize}

As before, all of the good orbits from class (B1) contribute to
\eqref{eq:c_me} and class (B2) makes no contribution. This is because,
by \eqref{eq:IRT1}, $\cz(\ga_i^{k_i-\l}) = d - \cz(\ga_i^\l)$, which
is $\leq d+1$ whenever $\cz(\ga_i^\l) \geq -1$ and $ > d+1$ whenever
$\cz(\ga_i^\l) < -1$.

At this point recall that when $n$ is even $\alpha$ is assumed to have
no contractible good periodic orbits of index $0$ or $\pm 1$. (As in
the case of an odd $n$, this is the key point where this assumption is
utilized.) Hence, for all good iterates $ \ga_i^{k_i+\l}$ in class
(C1), $\cz(\ga_i^\l) > 1$ and, by \eqref{eq:IRT2},
$\cz(\ga_i^{k_i+\l}) = d + \cz(\ga_i^\l) > d + 1$ whenever
$\ga_i^{k_i+\l}$ is good. As a result, class (C1) does not contribute
to \eqref{eq:c_me}.  Finally, all of the good orbits from class (C2)
contribute to \eqref{eq:c_me} since $\cz(\ga_i^\l) < -1$ and
$\cz(\ga_i^{k_i+\l}) = d + \cz(\ga_i^\l) < d-1$ by \eqref{eq:IRT2}.

Thus, as in Case 1, all of the good orbits from classes (A), (B1) and
(C2) have index $\leq d+1$ and contribute to \eqref{eq:c_me}, and good
orbits from classes (D), (B2) and (C1) have index $> d+1$ and make no
contribution to \eqref{eq:c_me}.  Similarly to \eqref{eq:c^e} and
\eqref{eq:c^o}, define
\begin{equation*}
\label{eq:c^e-even}
c^e_\pm = \sum_{i=1}^r \#\{1 \leq \l \leq \lo \mid 
\cz(\ga_i^\l) < -1,\ \ga_i^{k_i\pm \l}
\text{ is good and }\cz(\ga_i^\l)\text{ is even}\}
\end{equation*}
and
\begin{equation*}
\label{eq:c^o-even}
c^o_\pm = \sum_{i=1}^r \#\{1 \leq \l \leq \lo \mid 
\cz(\ga_i^\l) < -1,\ \ga_i^{k_i\pm \l}
\text{ is good and }\cz(\ga_i^\l)\text{ is odd}\}.
\end{equation*}
Consider again $ \sum_{m=m_{\min}}^\infty (-1)^m c_m $ to which all
good orbits from classes (A)--(D) and the collection
$\{\gamma_i^{k_i}\}$ contribute. In particular, contributions of
classes (B2) and (C2) are, respectively, $c^e_- - c^o_-$ and
$c_+^e-c_+^o$. With the above discussion in mind, viewing
\eqref{eq:c_me} as
$ \sum_{m=m_{\min}}^\infty (-1)^m c_m - \sum_{m>d+1} (-1)^m c_m$, we
obtain
\begin{equation*}
\label{eq:contributions1-even}
\sum_{m=m_{\min}}^{d+1} (-1)^m c_m 
= \sum_{i=1}^r\bigg(\sum_{\l=1}^{k_i} \chi(\ga_i^\l) 
+ c^e_+-c^o_+ - (c^e_--c^o_-) 
- \sum_{\cz(\ga_i^{k_i})>d+1} \chi(\ga_i^{k_i})\bigg).
\end{equation*}
We again have $c^e_- = c^e_+$ and $ c^o_- = c^o_+$. Therefore,
\begin{equation}
\label{eq:contributions2-even}
\sum_{m=m_{\min}}^{d+1} (-1)^m c_m 
= \sum_{i=1}^r\sum_{\l=1}^{k_i} \chi(\ga_i^\l) 
- \sum_{i=1}^r\sum_{\cz(\ga_i^{k_i})>d+1} \chi(\ga_i^{k_i}).
\end{equation}
Similarly to \eqref{eq:r^e} and \eqref{eq:r^o}, define
\begin{equation}
\label{eq:r^e-even}
r^e_\pm = \#\{1 \leq i \leq r \mid 
\pm(\cz(\ga_i^{k_i}) - d) > 1,\ \ga_i^{k_i}
\text{ is good and }\cz(\ga_i^{k_i})\text{ is even}\}
\end{equation}
and
\begin{equation*}
\label{eq:r^o-even}
r^o_\pm = \#\{1 \leq i \leq r \mid 
\pm(\cz(\ga_i^{k_i}) - d) > 1,\ \ga_i^{k_i}
\text{ is good and }\cz(\ga_i^{k_i})\text{ is odd}\}.
\end{equation*}
Notice that the last term in \eqref{eq:contributions2-even} is given
by $r^e_+ - r^o_+$.  Then, with the assumption $n$ is even in mind,
setting $d=2sc_B$ and using Lemma \ref{lemma:resonance}, we turn
\eqref{eq:contributions2-even} into
\begin{align*}
\label{eq:morse_ineq-even}
s\chi(B) - r^e_+ + r^o_+ 
& = \sum_{m=m_{\min}}^{d+1} (-1)^m c_m \\
& \leq \sum_{m=m_{\min}}^{d+1} (-1)^m b_m \\
& = s\chi(B) - \sum_{i=0}^{n-2} (-1)^i \dim H_i(B;\Q).
\end{align*}
The inequality is due to the Morse inequalities \eqref{eq:Morse_ineq}
with the direction reversed since $d+1$ is odd, and the last equality
follows from \eqref{eq:CH} using the hypothesis that $c_B>n/2$.  Hence
\begin{equation*}
\label{eq:lower estimate r^e_+}
r^e_+ \geq \sum_{i=0}^{n-2} (-1)^i \dim H_i(B;\Q).
\end{equation*}
Arguing similarly to the case where $n$ is odd, it is not hard to see
that we also have 
\begin{equation*}
\label{eq:lower estimate r^e_-}
r^e_- \geq \sum_{i=0}^{n-2} (-1)^i \dim H_i(B;\Q).
\end{equation*}
Since $r^e_+$ and $r^e_-$ correspond to two disjoint collections of
simple orbits, these two inequalities imply that we must have at least
$j$ distinct contractible simple closed orbits, say,
$\ga_1,\dots,\ga_{j}$, where
\begin{align}
\label{eq:lower estimate q}
j & := 2 \sum_{i=0}^{n-2} (-1)^i \dim H_i(B;\Q) \nonumber \\
  & = \chi(B) + 2 \dim H_{n-1}(B;\Q) - \dim H_n(B;\Q).
\end{align}
Here the equality is due to Poincar\'e duality and the assumption that
$n$ is even.  Observe that iterates of these orbits do not contribute
to $\HC^0_\ast(M)$ in degrees $\ast = d, d \pm 1$. Indeed, since
$\cz(\ga_i^\l) \notin \{-1,0,1\}$ for every $1 \leq i \leq r$ and
$\l \in \N$ such that $\ga_i^\l$ is good, we infer from
\eqref{eq:IRT1}, \eqref{eq:IRT2} and the definition of $\lo$ that
\[
\cz(\ga_i^\l) \notin \{d-1,d,d+1\}
\]
for all $\l \neq k_i$ and $1 \leq i \leq r$ such that $\ga_i^\l$ is
good. Thus only the orbits $\ga_1^{k_1},\dots,\ga_r^{k_r}$ can
contribute to $\oplus_{m=d-1}^{d+1} \HC^0_{m}(M)$. However, it follows
from the definition of $r^e_\pm$ given by \eqref{eq:r^e-even} that
$\cz(\ga_i^{k_i}) \notin \{d-1,d,d+1\}$ for every $1 \leq i \leq
j$. Hence
\begin{align}
\label{ineq:r-even}
r & \geq j + \sum_{m=d-1}^{d+1} \dim \HC^0_{m}(M).
\end{align}
By \eqref{eq:CH} and Poincar\'e duality, we also have
\begin{align}
\label{ineq:hc^0-even}
\sum_{m=d-1}^{d+1} \dim \HC^0_{m}(M) 
& \geq \sum_{m=n-1}^{n+1} \dim H_m(B;\Q) \nonumber\\
& = 2\dim H_{n-1}(B;\Q) + \dim H_n(B;\Q).
\end{align}
Finally, combining \eqref{eq:lower estimate q}, \eqref{ineq:r-even}
and \eqref{ineq:hc^0-even}, we obtain
\begin{align*}
  r \geq \chi(B) + 4\dim H_{n-1}(B;\Q), 
\end{align*}
which establishes Theorem \ref{thm:main} when $n$ is even.

\begin{remark}
\label{rmk:non-hyp2}
By item \ref{cond:i} of Theorem \ref{thm:IRT}, the above argument
shows that if $n$ is even and $\alpha$ has finitely many distinct
contractible simple closed orbits, then at least
$\chi(B) + 4\dim H_{n-1}(B) - \dim H_{n}(B)$ of them are
non-hyperbolic. This proves Theorem \ref{thm:non-hyp} when $n$ is
even.
\end{remark}

\subsection{Proof of Theorem \ref{thm:app}}
\label{sec:proof cor2}
By Proposition \ref{prop:CH}, $B$ is necessarily positive or negative
spherically monotone.  We will prove the theorem in the positive
monotone case; the argument in the negative monotone case is similar.
Arguing by contradiction, assume that $\alpha$ has only one
contractible simple closed orbit $\ga$. Note that the assumption
$c_B>n/2$ implies that $2c_B\geq n+1$ if $n$ is odd and $2c_B\geq n+2$
if $n$ is even. Therefore, by the isomorphism \eqref{eq:CH},
$\HC^0_m(M)=0$ for every $m<1$ if $n$ is odd or $m<2$ if $n$ is
even. Moreover, there exists a sequence $m_i\to\infty$ such that
$\HC_{m_i}(M)\neq 0$ for every $i$. As consequence, $\mi(\ga)>0$ and
every good iterate of $\ga$ must have index $\geq 1$ if $n$ is odd or
$\geq 2$ if $n$ is even. This implies that $\alpha$ is index-positive
and has no good contractible closed orbits $\ga^k$ such that
$\cz(\ga^k)=0$ if $n$ is odd or $\cz(\ga^k) \in \{0,\pm 1\}$ if $n$ is
even. Thus $M$ and $\alpha$ satisfy condition \ref{cond:F} and so
Theorem \ref{thm:main} applies. This contradicts the assumption that
$\alpha$ has only one contractible simple closed orbit.

\section{Multiplicity results and the contact Conley conjecture via
  contact homology}
\label{sec:sft}

In this section we discuss a generalization of Theorem \ref{thm:main}
relying on a variant of hybrid cylindrical-linearized contact
homology. As another application of this homological construction we
state a refinement of the contact Conley conjecture proved in
\cite{GGM}.

\subsection{Contact homology} 
\label{sec:ch}
Let $(M^{2n+1},\xi)$ be a contact manifold and let $\alpha$ be a
non-degenerate contact form supporting $\xi$. For the sake of
simplicity, we assume that $c_1(\xi)=0$.  The differential graded
algebra $(\AC(M,\alpha),d_\alpha)$ underlying the full rational
contact homology is a graded commutative algebra generated by good
closed Reeb orbits of $\alpha$; see \cite{Bo:survey, EGH}.  With our
dimension conventions, the grading is given by
$|\gamma|=\mu(\gamma)+n-2$. Assume furthermore that $(M,\xi)$ admits a
non-degenerate index-admissible contact form $\beta$. Without loss of
generality, we may assume that $\beta<\alpha$, i.e., $\beta=f\alpha$
where $0<f<1$. Hence we have a cylindrical cobordism from $(M,\beta)$
to $(M,\alpha)$ in the symplectization of $M$, resulting in a
homomorphism
$\Phi_\beta\colon(\AC(M,\alpha),d_\alpha)\to (\AC(M,\beta),d_\beta)$
of differential graded algebras.

Since $\beta$ is index-admissible, $(\AC(M,\beta),d_\beta)$ has a
unique ``trivial'' augmentation $\eps_0$ determined by the requirement
that the only monomial of degree zero for which $\eps_0\neq 0$ is
$1$. Composing $\eps_0$ with $\Phi_\beta$, we obtain the augmentation
$$
\eps_\beta =\eps_0\circ \Phi_\beta\colon
(\AC(M,\alpha),d_\alpha)\to\Q.
$$
Note that, as is easy to see, $\eps_\beta(\gamma)=0$ whenever $\gamma$
is not contractible.

A routine argument shows that the linearized homology of
$(\AC(M,\beta),d_\beta)$ with respect to $\eps_\beta$ is independent
of $\beta$ and $\alpha$; cf.\ \cite{Bo:survey, EGH}. This is the
``hybrid'' homology we will use in this section but, for the sake of
simplicity, we will still refer to this homology as the cylindrical
contact homology of $\alpha$. The main advantage of this construction
over the standard cylindrical contact homology is that the homology is
defined for all non-degenerate contact forms: the form $\alpha$ need
not be index-admissible. The only requirement is that $(M,\xi)$ admits
one index-admissible form. It is essential that this homology can
still be viewed as the homology of a complex freely generated by good
closed Reeb orbits of $\alpha$. The complex is graded by $|\gamma|$ or
$\mu(\gamma)$ and filtered by the contact action. Furthermore -- and
this is essential for what follows -- the complex is also graded by
the free homotopy class of $\gamma$ just as the standard cylindrical
contact homology complex. (This is a consequence of the fact that
$\eps_\beta$ vanishes on non-contractible orbits.)

\begin{remark}
  The foundational aspects of the contact homology theory are still to
  be fully laid down. We refer the reader to \cite{HWZ10, HWZ11,
    HWZ14} for the polyfold approach to this theory and to
  \cite{Par:GT, Par} for the virtual cycle approach and further
  references.
\end{remark}

\subsection{Multiplicity results}
\label{sec:ch-m}
Using contact homology we have the following refinement of Theorem
\ref{thm:main}.

\begin{theorem}
\label{thm:main-ch}
Let $(M^{2n+1},\xi)$ be a prequantization $S^1$-bundle of a closed
symplectic manifold $(B,\om)$ such that $\om|_{\pi_2(B)}\neq 0$ and
$c_1(\xi)=0$. Then $(B,\omega)$ is necessarily monotone and we require
that $c_B>n/2$ when it is positive monotone and $c_B\geq n$ when it is
negative monotone. Assume, furthermore, that $H_{k}(B;\Q)=0$ for every
odd $k$ or $c_B>n$. Let $\alpha$ be an index-definite non-degenerate
contact form on $(M,\xi)$ having no contractible good periodic orbits
$\ga$ with $\cz(\ga)=0$ if $n$ is odd or with
$\cz(\ga) \in \{0,\pm 1\}$ if $n$ is even. Then $\alpha$ carries at
least $r_B$ geometrically distinct contractible periodic orbits, where
\begin{equation*}
r_B:=
\begin{cases}
\chi(B) + 2\dim H_n(B;\Q)
\text{ if } n\text{ is odd} \\
\chi(B) + 4\dim H_{n-1}(B;\Q)
\text{ if } n\text{ is even.}
\end{cases}
\end{equation*}
\end{theorem}

\begin{proof}[A few words about the proof] The requirement that
  $c_1(\xi)=0$ guarantees via the Gysin exact sequence that
  $(B,\omega)$ is monotone, i.e., $c_1(TB)=\lambda [\om]$ in
  $H^2(B;\Q)$ for some $\lambda\in\R$; cf.\ Remark \ref{rmk:monotone}.
  As in the proof of Proposition \ref{prop:CH}, we first need to show
  that the contact homology of $(M,\xi)$ is defined and given by
  \eqref{eq:CH} or \eqref{eq:CH-nm} depending on whether $B$ is
  positive ($\lambda\geq 0$) or negative ($\lambda<0$) monotone.  To
  this end, it is sufficient to show that a non-degenerate
  perturbation $\beta$ of a connection contact form $\alpha_0$ is
  index-admissible.

  When $B$ is positive monotone, this is an immediate consequence of
  Lemma \ref{lemma:pos_index} together with the requirement that
  $c_B>n/2$.  (In fact, it is enough to assume that $c_B\geq 2$ which
  follows from $c_B>n/2$ when $n\geq 2$ and holds automatically under
  the conditions of the theorem when $n=1$ since $c_{S^2}=2$.) Then
  \eqref{eq:CH} follows exactly as in the proof of Proposition
  \ref{prop:CH} using now the Morse--Bott calculations of contact
  homology from \cite{Bo:thesis}.

  When $B$ is negative monotone, the situation is similar. In the
  notation from the proof of Lemma \ref{lemma:pos_index}, for every
  $T_0>0$, all closed Reeb orbits $\gamma$ of a sufficiently small
  perturbation $\beta$ of $\alpha_0$ with $T(\gamma)\leq T_0$ have
  $\mu(\gamma)\leq n-2c_B$. Then, an argument completely similar to
  the proof of the lemma and using the quasimorphism property of the
  mean index shows that the same is true for all orbits when $\beta$
  is sufficiently closed to $\alpha_0$. Thus $\beta$ is
  index-admissible when $n-2c_B<1-n$, or equivalently $c_B\geq n$, and
  \eqref{eq:CH-nm} follows again from the results in \cite{Bo:thesis}
  and the action filtration spectral sequence.

  The proof is then finished exactly in the same way as the proof of
  Theorem \ref{thm:main}; see Section~\ref{sec:pf-main}. (In the SFT
  framework developed in \cite{Par}, the gap in this argument is a
  Morse--Bott calculation of contact homology similar to
  \cite{Bo:thesis} or \cite{Poz}.)
\end{proof}

There is a broad class of symplectic manifolds to which Theorem
\ref{thm:main-ch} applies, while Theorem \ref{thm:main} does
not. Among these are, for instance, negative monotone symplectic
manifolds with large $c_B$, e.g., complete intersections of high
degree. These manifolds can have $H_{\mathit{odd}}(B;\Q)\neq 0$. A
simple example is the product of a complete intersection of a
sufficiently high degree and a symplectically aspherical manifold.

\begin{remark}
  It is clear from the discussion above that there is also a
  refinement of Theorem \ref{thm:non-hyp} relying on contact homology.
  Namely, under the assumptions of Theorem \ref{thm:main-ch}, if the
  contact form $\alpha$ has finitely many geometrically distinct
  contractible closed orbits then it carries at least
  $r^\text{non-hyp}_B$ geometrically distinct contractible
  non-hyperbolic periodic orbits.
\end{remark}

\subsection{Contact Conley conjecture}
\label{sec:ccc}
Another application of our definition of the cylindrical contact
homology is a refinement of the contact Conley conjecture originally
proved in \cite{GGM}. Namely, we have

\begin{theorem}[Contact Conley Conjecture] 
\label{thm:ccc}
Let $M\to B$ be a prequantization bundle and let $\alpha$ be a contact
form on $M$ supporting the standard (co-oriented) contact structure
$\xi$ on $M$.  Assume that
\begin{itemize}
\item[\rm{(i)}] $B$ is aspherical, i.e., $\pi_r(B)=0$ for all
  $r\geq 2$, and
\item[\rm{(ii)}] $c_1(\xi)\in H^2(M;\Q)$ is atoroidal.
\end{itemize}
Then the Reeb flow of $\alpha$ has infinitely many simple closed
orbits with contractible projections to $B$. Assume furthermore that
the Reeb flow has finitely many closed Reeb orbits in the free
homotopy class $\ff$ of the fiber and that these orbits are weakly
non-degenerate. Then for every sufficiently large prime $k$ the Reeb
flow of $\alpha$ has a simple closed orbit in the class $\ff^k$.
\end{theorem}

The new point here, as compared to \cite[Thm.\ 2.1]{GGM}, is that the
form $\alpha$ is not required to be index-admissible. The proof of
Theorem \ref{thm:ccc} is essentially identical to the proof of
\cite[Thm.\ 2.1]{GGM} and the only difference is that with our
definition of cylindrical contact homology we need to ensure the
existence of just one non-degenerate index-admissible contact form. In
fact, every sufficiently small non-degenerate perturbation $\beta$ of
a connection contact form $\alpha_0$ on $M\to B$ is index
admissible. This is an immediate consequence of, e.g., \cite[Prop.\
3.1]{GG:MW} asserting that for every contractible closed Reeb orbit
$\gamma$ of $\beta$, we have
$$
\big|\hmu(\gamma)\big|\geq O\big(T(\gamma)\big),
$$
where $T(\gamma)$ is the period of $\gamma$. (The proof of this fact
is somewhat similar to the proof of Lemma \ref{lemma:pos_index}. As
in the proof of Theorem \ref{thm:main-ch}, in the contact homology
framework from \cite{Par} the missing part in this argument is a
Morse--Bott calculation of contact homology.)

\end{document}